\documentclass[12pt,a4paper]{amsart}
\usepackage[utf8]{inputenc}
\usepackage{a4wide,amssymb,xcolor,array,hyperref}

\hypersetup{
 colorlinks,
 linkcolor={blue!90!black},
 citecolor={red!80!black},
 urlcolor={blue!50!black}
}

\newcommand{\R}{\mathbb{R}}

\newcommand{\Sd}{\mathbb{S}^d}

\newcommand\bigO[1]{\mathcal{O}(#1)}

\DeclareMathOperator{\var}{\mathbb{V}ar}

\DeclareMathOperator{\Div}{div}

\newcommand{\vertiii}[1]{{\left\vert\kern-0.25ex\left\vert\kern-0.25ex\left\vert #1 
    \right\vert\kern-0.25ex\right\vert\kern-0.25ex\right\vert}}

\theoremstyle{plain}
\newtheorem{thm}{Theorem}[section]

\newtheorem{lem}[thm]{Lemma}
\newtheorem{prop}[thm]{Proposition}
\theoremstyle{definition}

\begin{document}

\title[DPPs and norm representations]{Linear statistics of determinantal point processes and norm representations}

\author{Matteo Levi}
\address{Fac. de Ciències, Universitat Autònoma de Barcelona, 08193 Bellaterra, Spain, and Dept.\ Matem\`atica i Inform\`atica,
Universitat  de Barcelona,
Gran Via 585, 08007 Bar\-ce\-lo\-na, Spain}
\email{\href{mailto: matteo.levi@ub.edu}{\texttt{matteo.levi@ub.edu}}}

\author{Jordi Marzo}
\address{Dept.\ Matem\`atica i Inform\`atica,
 Universitat  de Barcelona and BGSMath,
Gran Via 585, 08007 Bar\-ce\-lo\-na, Spain and
CRM, Centre de Recerca Matemàtica, Campus de Bellaterra Edifici C, 08193 Bellaterra, Barcelona, Spain}
\email{\href{mailto:jmarzo@ub.edu}{\texttt{jmarzo@ub.edu}}}

\author{Joaquim Ortega-Cerd\`a}
\address{Dept.\ Matem\`atica i Inform\`atica,
 Universitat  de Barcelona,
Gran Via 585, 08007 Bar\-ce\-lo\-na, Spain and
CRM, Centre de Recerca Matemàtica, Campus de Bellaterra Edifici C, 08193 Bellaterra, Barcelona, Spain}
\email{\href{mailto:jortega@ub.edu}{\texttt{jortega@ub.edu}}}

\thanks{The authors have been partially supported by grants  PID2021-123405NB-I00, PID2021-123151NB-I00 by the Ministerio de Ciencia e Innovación and by the Departament de Recerca i Universitats, grant 2021 SGR 00087}

\begin{abstract}
  We study the asymptotic behavior of the fluctuations of smooth and rough linear statistics for determinantal point processes on the sphere and on the Euclidean space. The main tool is the generalization of some norm representation results for functions in Sobolev spaces and in the space of functions of bounded variation.
\end{abstract}

\maketitle

\section{Introduction}

Let $f$ be a smooth function defined on a compact Riemannian manifold $\mathbb  M$. The objective of the Monte Carlo method is to approximate the value of the integral 
$\int_{\mathbb  M} f dV$ by the empirical average 
\[
\frac{1}{N} \sum_{i=1}^N f(x_i),
\] 
where the points $x_1,\dots , x_N$ are chosen randomly with respect to some distribution in $\mathbb  M$ and $V$ is the normalized volume form induced by the metric in $\mathbb M$. A particularly simple setting is to consider $x_1,\dots , x_N$ i.i.d. uniform points distributed according to $V$; in this case one can assure almost sure convergence of the empirical average  to the integral by the strong law of large numbers. Quantitatively, when the function is in $L^2(\mathbb  M)$, one can deduce, from Jensen's inequality and the computation
\begin{equation}\label{variance_lin}
\var\Big(   \frac{1}{N}\sum_{i=1}^N f(x_i)\Big)=\frac{1}{N} \int_{\mathbb  M} \Big|f(x)-\int_{\mathbb  M} fdV\Big|^2\,  dV(x),
\end{equation}
the following rate of approximation 
\[
\mathbb{E}\left( \left| \frac{1}{N}\sum_{i=1}^N f(x_i)-\int_{\mathbb  M} f(x) \, dV(x) \right| \right) \le \frac{ \| f-\int_{\mathbb  M} f\, dV \|_{2}}{\sqrt{N}}.
\]

This shows already one of the limitations of the classical Monte Carlo technique, the $N^{-1/2}$ convergence rate. Of course, numerical integration with random nodes has been studied extensively, and the convergence rate above can be improved by different methods, see for example \cite{Lem09,LP14}. One of the recently proposed methods consists in the use of 
repulsive point processes \cite{BH20}. Indeed, the authors in \cite{BH20} sample the nodes of integration from a determinantal point process in $[-1,1]^d$ defined by an orthogonal polynomial ensemble and show that the rate of convergence can be improved to 
$N^{-(1+\frac{1}{d})/2}$. Determinantal point processes form a particular class of random point processes where the points present a built-in repulsion. One of the reasons for their current popularity is their appearance in random matrix theory, but they have been studied in connection with many problems. In this work we determine the convergence rate of the Monte Carlo integration, for some determinantal point processes on the $d$-dimensional sphere $\mathbb S^d$ and the Euclidean space, which have been recently studied in connection with problems of discrete energy minimization, hyperuniformity, discrepancy or optimal transport \cite{AZ15,BMOC16,BGKZ20,Ste20,ADGMS23,BGM23}.

Given a random point process $\Xi$ in a manifold $\mathbb  M$, i.e., a random discrete subset $\{x_i\}_i \subset \mathbb  M$, and a measurable function $\phi:\mathbb  M\rightarrow \mathbb C$, the corresponding linear statistic is the random variable
\[
\Xi(\phi)=\int_{\mathbb  M} \phi \;d\; \Xi=\sum_i \phi(x_i).
\]
When $\phi =\chi_A$ is the characteristic function of a Borel set $A\subset \mathbb  M$, then $\Xi(\chi_A)$, the number of points of the process $\Xi$ in $A$, will be denoted as $n_A$. Characteristic functions define rough linear statistics,
and when $\phi$ enjoys some degree of smoothness, namely, if it belongs to some (possibly fractional) Sobolev space of big enough order, we talk about smooth linear statistics.

As in the computation above, \eqref{variance_lin}, to control the convergence rate we study the  second moment of the linear statistics. In order to do so, we observe that the expression of the variance of a linear statistic for a determinantal point process can be related with the norm of the function in a Sobolev space (in the smooth case) or in the space of functions of bounded variation (in the rough case). By means of some norm representation results of Bourgain, Brezis and Mironescu type \cite{BBM}, the problem of computing the asymptotics of the linear statistics will follow from the realization of a particular sequence of functions as mollifiers. This approach is not working in some cases because the functions involved are not mollifiers (this happens, for instance, for the harmonic ensemble) and we are led to estimate quadratic forms in terms of fractional Sobolev norms, a problem studied before by different authors, see \cite[Chapter~2]{FRRO}.

\textbf{Notation} In the remaining of the paper,  we write $A\lesssim B$ if there exists a positive constant $c$ such that $A\leq c B$, and $A\approx B$ if it is both $A\lesssim B$ and $B\lesssim A$.

We write $\sigma$ for the normalized surface measure on the sphere, and set $\omega_{d}:=\mathcal{H}^d(\mathbb S^d)=
2\pi^{\frac{d+1}{2}}/ \Gamma\left( \frac{d+1}{2} \right).$

\subsection{Determinantal point processes}

A simple point process $\Xi$ on a complete and separable metric space $X$ is called a determinantal point process, DPP for short, if its joint intensities, with respect to some background Radon measure $\mu$, are given by
\[
\varrho_{\ell} (x_1,\dots , x_{\ell})=\det( K(x_i,x_j))_{i,j=1}^{\ell},
\]
for $\ell=1,2,\dots$,
where $K:X\times X\longrightarrow \mathbb C$ is a locally square integrable kernel function such that the associated integral operator
is Hermitian, non-negative definite and locally of trace class. DPPs were introduced by Macchi in the 70s as a mathematical model for fermions in quantum mechanics, \cite{Mac75,Mac77}, and have been widely studied in many settings, see \cite{AGZ10,HKPV09} for references and background. 

Throughout the paper we are going to consider only DPPs of projection type, i.e., those such that the corresponding integral operator is bounded and all its non-zero eigenvalues are 1. In this case the expression of the linear statistics is particularly simple. Indeed,
for any bounded measurable function $f$ defined in $X$, the expectation is given by
\[
\mathbb E \left( \sum_{x\in \Xi}f(x) \right) =\int_{X}K(x,x)f(x)\, d\mu(x),
\]
and the fluctuations are
\begin{equation}\label{equation_variance}
\var\left(  \sum_{x\in \Xi}f(x) \right)=\frac{1}{2}
\int_{X}\int_{X}|f(x)-f(y)|^2 |K(x,y)|^2 \, d\mu(x)d\mu(y).
\end{equation}
Now we define the determinantal point processes that we study.

\subsubsection{Harmonic ensemble}
In a smooth compact manifold $\mathbb M$, the spectrum of the Laplace-Beltrami operator $\Delta$ is discrete and there exists a sequence
\[
0\le \lambda_1^2\le \lambda_2^2\le \dots \to +\infty,
\]
and an orthonormal basis in $L^2(V)$ of eigenfunctions $\phi_i$ such that $\Delta \phi_i=-\lambda_i^2 \phi_i$. Given $L\ge 0$ we denote by $E_L$ the space generated by eigenfunctions of eigenvalue at most $L^2.$
The harmonic ensemble is the projection DPP defined by the reproducing kernel 
\[
K_L(x,y)=\sum_{i=1}^{k_L}\phi_i(x)\phi_i(y),\quad x,y \in \mathbb M,
\]
where $k_L=\dim (E_L)$. Observe that the corresponding integral operator is a projection onto the space $E_L.$

According to the definition above, this projection DPP with finite rank is defined by the joint intensities with respect to the volume form
\begin{equation}\label{def_dpp}
  \varrho_{\ell} (x_1,\dots , x_{\ell})=\det( K_L(x_i,x_j))_{i,j=1}^{\ell},  
\end{equation}
for $\ell\ge 1$. Observe that the joint intensities are zero for $\ell>k_L$ because the kernel has rank $k_L$ and  
the process has $\mathbb E (n_{\mathbb M})=k_L$ simple points almost surely.

In particular we consider the harmonic ensemble on the d-dimensional sphere $\mathbb M=\Sd\subset \R^{d+1}$, with its normalized surface measure, which we denote by $\sigma$. For an integer $\ell\ge 0$, let $\mathcal{H}_\ell$ be the vector space of 
the spherical harmonics of degree $\ell$, that is, the space of eigenfunctions of the Laplace-Beltrami operator with eigenvalue $\lambda_\ell=\ell (\ell +d-1)$,
\[
-\Delta Y=\lambda_\ell Y,\quad \quad  Y\in \mathcal{H}_{\ell}.
\]
The space of spherical harmonics of degree at most 
$L$, denoted by
$\Pi_L=\bigoplus_{\ell=0}^L \mathcal{H}_\ell$,
is also the space of polynomials of degree at most $L$ in 
$\R^{d+1}$ restricted to $\mathbb{S}^d$.  
One has that $L^2(\mathbb{S}^d)=\bigoplus_{\ell\ge 0} \mathcal{H}_\ell$ and the Fourier series expansion of a function $f \in L^2(\mathbb{S}^d)$ is given by
\[
f=\sum_{\ell,k} f_{\ell,k} Y_{\ell,k},\qquad f_{\ell,k}=\langle f,Y_{\ell,k} \rangle=\int_{\mathbb{S}^d} f\, Y_{\ell,k} \,d\sigma,
\]
where $\{ Y_{\ell,k} \}_{k=1}^{h_\ell}$ is an orthonormal basis of $\mathcal{H}_\ell$.

The harmonic ensemble on $\Sd$ for $d\ge 2$ (see \cite{BMOC16}) is defined by \eqref{def_dpp} with the reproducing kernel of the space $\Pi_L$ for $L\ge 0$
\[
    K_L(x,y)=\frac{\pi_L}{\binom{L+d/2}{L}}P_L^{(d/2,d/2-1)}(\langle x, y \rangle),
\]
where
\begin{equation}\label{eq:piL}
\pi_L:=\dim \Pi_L=\frac{2L+d}{d}\binom{d+L-1}{L} = 
\frac{2}{\Gamma(d+1)}L^d+o(L^d),
\end{equation} 
and the Jacobi polynomials 
$P_L^{(d/2,d/2-1)}(t)$ are normalized as 
\[
P_L^{(d/2,d/2-1)}(1)=\binom{L+d/2}{L}=\frac{\Gamma\left(L+d/2+1\right)}{\Gamma(L+1)\Gamma\left(d/2+1\right)}.
\]
Observe that $d(x,y)=\arccos{\langle x,y\rangle}$ is the geodesic distance on $\mathbb{S}^d$ and therefore the kernel is rotational invariant and the corresponding DPP is an isotropic process with $N=\pi_L$ points. When 
$d=1$ this DPP coincides with the Circular Unitary Ensemble (CUE).

\subsubsection{Spherical ensemble}

Given $A,B$ independent $N\times N$ random matrices with i.i.d. complex standard entries, the eigenvalues $z_1,\dots ,z_N\in \mathbb C$ of $A^{-1}B$ form a DPP in $\mathbb C$ with kernel $(1+z\overline{w})^{N-1}$ with respect to the 
background measure $\frac{N}{\pi(1+|z|^2)^{N+1}}dz$, \cite{Kri09}.  Then, if $g(x)=z$ is the stereographic projection of the sphere $\mathbb S^2$ from the North Pole onto the complex plane, the spherical ensemble can be defined via the correlations \eqref{def_dpp} with the kernel
\[
K_N(x,y)=N \frac{(1+z \bar{w})^{N-1}}{ (1+|z|^2)^{\frac{N-1}{2}} (1+|w|^2)^{\frac{N-1}{2}}  },
\]
for $g(x)=z,g(y)=w\in \mathbb C.$
Observe that
\[
|K_N(x,y)|^2=N^2 \left( 1-\frac{|x-y|^2}{4}  \right)^{N-1},
\]
i.e., the second intensity is rotation invariant, $K_N(x,x)=N$ and the DPP has $N$ points.

\subsubsection{Ginibre point process}\label{sec:ginibre}
The finite Ginibre ensemble is the process in $\mathbb C$ whose points are the eigenvalues of an ensemble of random $N\times N$ matrices having as entries i.i.d. standard complex Gaussians. Taking the Lebesgue measure as background measure, the kernel is
\[
\frac{1}{\pi}e^{-\frac{|z|^2+|w|^2}{2}}\sum_{k=0}^{N-1}\frac{( z\overline{w})^k}{k!}.
\]
Observe that as $N\to\infty$ the kernel converges to $\pi^{-1}e^{-\frac{|z-w|^2}{2}}$, which defines a determinantal point process invariant under translations with first intensity one.

We will study here a rescaled version of the finite Ginibre ensemble, namely, we renormalize the Gaussian entries of the  matrices to have variance $1/N$. The kernel then becomes
\begin{equation}\label{kernel_ginibre_finite}
    K_N(z,w)=\frac{N}{\pi}e^{-N\frac{|z|^2+|w|^2}{2}}\sum_{k=0}^{N-1}\frac{(N z\overline{w})^k}{k!}.
\end{equation}
The associated process is the same one that is studied, for instance, in \cite{virag}. With this normalization, the density of the process $K_N(z,z)/N$ converges to $\pi^{-1}\chi_{\mathbb D}+(2\pi)^{-1}\chi_{\partial \mathbb D}$ pointwise (simply use the asymptotic formula \cite[\href{https://dlmf.nist.gov/8.11.13}{(8.11.13)}]{NIST:DLMF}), which agrees with the fact that the eigenvalues lay in the unit disc with probability one, as prescribed by the circular law.

Analogously, rescaling the kernel $\pi^{-1}e^{-\frac{|z-w|^2}{2}}$ by a factor $1/L$ we obtain the radial kernel
\begin{equation}\label{kernel_ginibre}
\mathcal{K}_L(z,w)= \frac{L}{\pi}e^{-L\frac{|z-w|^2}2},
\end{equation}
which defines an infinite process on $\mathbb C$ with first intensity $L/\pi$, with the Lebesgue measure as underlying measure, which we call the infinite Ginibre process. 

\subsubsection{Bessel point process} The Bessel point process on $\mathbb R^d$ is the infinite rank DPP defined by the kernel
\[
K(x,y) =2^{d/2} \Gamma\left(\frac{d}{2}+1\right) \frac {J_{d/2}(\pi|x-y|)} {(\pi|x-y|)^{d/2}},
\]
where $J_\alpha$ is the Bessel function of the first kind, and with the Lebesgue measure in $\R^d$ as background measure. Observe that this function is the reproducing kernel of the Paley-Wiener space of $L^2$ functions bandlimited to the unit ball, and it is a natural generalization of the sine process defined on $\mathbb R$ to higher dimensions. This kernel defines a determinantal point process invariant under translations with intensity one, $K(x,x) = 1$. If we rescale it by a factor $L$ we get a process with intensity $L^d$ that has the corresponding kernel
\begin{equation}\label{kernel_bessel}
    K_L(x,y) = \Psi_L(|x-y|):=2^{d/2} \Gamma\left(\frac{d}{2}+1\right) \frac {L^{d/2}J_{d/2}(L\pi|x-y|)} {(\pi|x-y|)^{d/2}}.
\end{equation}
This process is, in some sense, universal, as it arises as the limit of the harmonic ensemble in any compact Riemannian manifold of dimension $d$, see \cite{KS2022}.

A particularly well studied case is when $d=1$, where we get the sine process in $\R$ with the cardinal sine kernel,
\[
K_L(x,y) =  \frac {\sin L\pi(x-y)}{\pi( x-y)},
\]
see \cite{For10}.

\subsection{Function spaces and norm representations}\label{sec:function spaces}
Let $\mathbb M$ be a $d$-dimensional smooth, complete Riemannian manifold without boundary. We let $d(\cdot,\cdot)$ denote the geodesic distance and $V$ the volume form, which we always assume to be normalized if $\mathbb M$ is compact.

We briefly introduce the main spaces of functions on $\mathbb M$ we will deal with, recalling those properties that are important to us. The first one is the Sobolev space 
\[W^{1,p}(\mathbb M)=\lbrace f\in L^1(\mathbb M): [ f ]_{1,p}^p<\infty\ \rbrace,\ p\geq 1.\] 
It is the natural generalization of the classical Sobolev space from the Euclidean to the Riemannian setting: the seminorm is given by
\[
    [ f ]_{1,p}^p=\int_{\mathbb M} |\nabla f|^p\, dV<\infty, \quad \text{if } f\in \mathcal C^1(\mathbb M),
\]
where for differentiable functions the gradient is defined as usual as the unique vector field $\nabla f$ such that $\langle \nabla f,X\rangle=df(X)$, for any smooth vector field $X$ on $\mathbb M$, while for non-differentiable functions it has to be interpreted in the weak sense. It is useful to keep in mind that $\mathcal C_c^\infty(\mathbb M)$ is dense in $W^{1,p}(\mathbb M)$, for every $p\geq 1$.

The second one is the fractional Sobolev space $W^{s,p}(\mathbb M)$, with $s\in (0,1)$ and $p\in [1,\infty)$, which is the space of functions in $L^p(\mathbb M)$ whose Gagliardo seminorm  $ [\,\cdot\,]_{s,p}$ is finite, i.e.,
\[
    [ f ]_{s,p}^p:=\iint_{\mathbb M\times \mathbb M}\frac{|f(x)-f(y)|^p}{d(x,y)^{d+sp}}\, dV(x)dV(y)<\infty.
\]
As it is customary, for any $s\in(0,1]$, in the Hilbert space case we write $H^s(\mathbb M)$ for $W^{s,2}(\mathbb M)$ and $ [\,\cdot\,]_s$ for $ [\,\cdot\,]_{s,2}$.

Finally, we define the space of functions of bounded variation on $\mathbb M$,
\[
    BV(\mathbb M):=\{f\in L^1(\mathbb M): [f]_{BV(\mathbb M)} <\infty\},
\]
where
\[
[f]_{BV(\mathbb M)} = \sup\left\{\int_{\mathbb M} f(x) \Div \phi(x) dV(x): \phi \in \mathcal C^\infty_c\text{ one-form with }\|\phi\|\le 1\right \}.
\]
 In particular, if $f$ is smooth, $[f]_{BV(\mathbb M)}=\int_{\mathbb M} |\nabla f|dV=[f]_{1,1}$. Smooth functions can be used to approximate bounded variation functions in the following sense \cite[Proposition 1.4]{MR2377131}: if a function $f$ belongs to $BV(\mathbb M)$, there exists a sequence of smooth functions such that
\begin{equation} \label{miranda approximation}
    \{f_j\}\subset\mathcal  C_c^\infty(\mathbb M), \quad f_j\to f \quad \text{in } L^1(\mathbb M), \quad \text{and} \quad \lim_j\int_{\mathbb M} |\nabla f_j|dV=[f]_{BV(\mathbb M)}.
\end{equation}
When studying the rough statistics of DPPs, the total variation of characteristic functions assume a particular importance. In analogy to the Euclidean setting, we say that a set $E\subset \mathbb M$ is a 
Caccioppoli set if $\chi_E\in BV(\mathbb M)$, and in this case we say that $[\chi_E]_{BV}$ is the perimeter of $E$. This terminology is justified because by the structure theorem of De Giorgi, 
$[\chi_E]_{BV} = \mathcal{H}^{d-1}(\partial_* E)$, where $\partial_* E$ is the measure theoretic boundary of $E$ and $\mathcal{H}^{d-1}$ is the $(d-1)$-Hausdorff measure. The measure theoretic boundary is defined as $\partial_* E = \mathbb M \setminus (I\cup O)$, where
\[
I= \left\{x\in \mathbb M: \lim_{r\to 0} \frac{V(B(x,r)\cap E)}{V(B(x,r))}  = 1\right\},
\]
and
\[
O= \left\{x\in \mathbb M: \lim_{r\to 0} \frac{V(B(x,r)\cap E)}{V(B(x,r))}  = 0\right\}.
\]
Of course, when $E$ has smooth boundary, the measure theoretic boundary coincides with the topological boundary.  For all these results we refer the reader to the books \cite{maggi,EG15}.

One key observation is to relate the study of the asymptotic behavior of the linear statistics of the determinantal point processes introduced in the previous section with some norm representation results for Sobolev and BV norms, see also \cite{Lin23} for a similar approach. Here and in the rest of the paper, $\{\varrho_n\}_{n\ge 0}$ is said to be a family of radial mollifiers on $\mathbb M$ if it is a sequence of functions $\varrho_n: (0,\infty)\to\mathbb{R^+}$ such that, for all $y\in \mathbb M$ and all $\delta>0$,
    \begin{equation}\label{def_mollifier}
          (i) \ \lim_{n\to\infty}\int_{\mathbb M} \varrho_n(d(x,y))dV(x)=1, \quad (ii) \ 
          \lim_{n\to\infty}\int_{B(y,\delta)^c} \varrho_n(d(x,y))dV(x)=0.
    \end{equation}
In the seminal paper \cite{BBM} Bourgain, Brezis and Mironescu proved that if $\Omega\subset \mathbb R^d$ is a bounded Lipschitz domain, $1< p<\infty$, $f\in L^p(\Omega)$ and $\{\varrho_n\}_{n\ge 0}$ is a family of radial mollifiers on $\mathbb R^d$, then
\[
        \lim_{n\to \infty}\iint_{\Omega \times\Omega}\frac{|f(x)-f(y)|^p}{|x-y|^p}\varrho_n(x-y)dx dy=K_{d,p}\int_{\Omega} |\nabla f|^p,
 \]
and that the same conclusion holds for $p=1$ if $f\in W^{1,1}(\Omega)$.
The constant $K_{d,p}$ is defined as
\[
    K_{d,p}=\int_{\mathbb{S}^{d-1}}|\langle e,\xi\rangle |^p d\sigma(\xi), \quad e\in \mathbb S^{d-1},
\]
and its value can be explicitly computed to be
\[
    K_{d,p}=\frac{\Gamma(d/2)\Gamma((p+1)/2)}{ \Gamma((d+p)/2)\sqrt{\pi}}.
\]
We set $K_d:=K_{d,1}=\omega_d/\sqrt{\pi}  \omega_{d-1}=\Gamma(d/2)/\sqrt{\pi} \Gamma((d+1)/2)$.

Answering to a question raised in \cite{BBM}, Dávila proved in \cite{Davila} that, under the same assumptions as above on the domain and on the mollifiers, if $f\in L^1(\Omega)$ then
\[
        \lim_{n\to \infty}\iint_{\Omega \times\Omega}\frac{|f(x)-f(y)|}{|x-y|}\varrho_n(x-y)dx dy=
            \displaystyle K_d [f]_{BV(\Omega)}.
 \]

Both the results of \cite{BBM} and of \cite{Davila} also hold if one takes $\Omega=\mathbb R^d$, see \cite{brezis}. These results have been recently extended to more general domains and to more general classes of mollifiers in $\mathbb R^d$ \cite{Gounoue}, as well as to the more general setting of Sobolev and BV spaces on Riemannian manifolds \cite{pegon, Davila-manifold}.

However, in most of the cases we consider in this paper, these previously known results are not directly applicable as they are, and that is the reason for which we are led to prove new norm representation type results, both on the unit sphere and in $\R^d$ .

\section{Main results}

The first natural question concerning the integration of functions by means of the empirical measure is the asymptotic equidistribution of the points. 

Our first result shows that points from the harmonic ensemble have, almost surely, spherical cap discrepancy converging to zero, and are therefore almost surely asymptotically equidistributed with respect to the volume form.

\begin{prop}\label{equidistribution}
    Let $\mathbb M$ be a compact Riemannian manifold and let $X_N=\{x_1,\dots , x_N\} \subset \mathbb M$ be $N$ points drawn from the harmonic ensemble. Then the spherical cap discrepancy
    \[
    D_\infty(X_N)=\sup_{D\subset M}\left| \frac{n_D}{N}-V(D)\right|,
    \]
    where the supremum runs through the geodesic balls in $\mathbb M$, satisfies
    \[
    \lim_{N\to +\infty} D_\infty(X_N)=0,
    \]
    with probability 1.
\end{prop}

The proof is similar to the one from Beck \cite{Bec84} (see also \cite{AZ15}). In fact, as it was observed in \cite{BGKZ20}, Beck in his result about bounds for discrepancies
used a particular DPP. From our proof, one can see that Proposition~\ref{equidistribution} is also true for DPPs with constant first intensity defined on Riemannian compact manifolds, such as the spherical ensemble. 

Our next result deals with the asymptotic behavior of smooth linear statistics for the harmonic ensemble in $\Sd$. Observe that we get the same 
order as in \cite{BH20}.

\begin{thm}\label{variance asymptotic}
    Let $f\in H^{1/2}(\Sd)$ with $d\ge 2$ and let $x_1,\dots , x_N$ be $N$ points drawn from the harmonic ensemble. Then
\begin{equation}\label{lim_smooth_harmonic}
\lim_{N\to +\infty} \frac{1}{N^{1-\frac{1}{d}}}\var\Big( \sum_{i=1}^N f(x_i)\Big)=\vertiii{f}_{1/2}^2\approx  [\,f\,]_{1/2}^2,
\end{equation}
and
\[
\mathbb{E}\left( \left| \frac{1}{N}\sum_{i=1}^N f(x_i)-\int_{\Sd} f(x) \, d\sigma(x) \right| \right) \lesssim \frac{[f]_{1/2}}{\sqrt{N^{1+\frac{1}{d}}}}.
\]
Moreover, for bounded $f\in H^{1/2}(\Sd)$ the following limit holds
\[
\sqrt{N^{1+\frac{1}{d}}}  \left( \frac 1N \sum_{i=1}^N f(x_i)-\int_{\Sd} f(x) \, d\sigma(x)\right) \xrightarrow{\text{law}} N(0,\vertiii{f}_{1/2}^2),
\]
where $\vertiii{\cdot}_{1/2}^2$ is a seminorm equivalent (up to the dimension $d$) to $ [\,\cdot\,]_{1/2}^2$.
\end{thm}

The first part of this result generalizes the upper bound for $\mathcal{C}^1$ functions proved in \cite{BMOC16}. The content of this theorem is not new, because it follows from known results. Indeed, for $d=1$ the harmonic ensemble on the sphere reduces to the CUE point process, and it was proved in \cite{Joh97} that the analogue of Theorem~\ref{variance asymptotic} follows from the strong Sz\"ego theorem. In fact, the relation of this kind of results with the strong Sz\"ego theorem was already observed by Dyson in the 60s, see the review  paper \cite{DIK13}. The strong Sz\"ego theorem was generalized, using the theory of pseudodifferential operators, to $d=2,3$ by Okikiolu \cite{Oki96}, and to the $d$-dimensional sphere (and even to Zoll manifolds), for all $d\geq 1$, by Okikiolu and Guillemin \cite{OG97}. As in dimension one, also in higher dimensions it is possible to write the moment generating function of a linear statistic as a Fredholm determinant and obtain the asymptotic behavior of the linear statistic proved in \eqref{lim_smooth_harmonic}. Once the limit \eqref{lim_smooth_harmonic} is established, the rest follows from the CLT for DPPs due to Soshnikov \cite{Sos02}, see also \cite{BH20}.

The novelty is then not the result itself, but the method of proof of \eqref{lim_smooth_harmonic}. We relate, through \eqref{equation_variance}, the asymptotics of the linear statistics with estimates for quadratic forms in terms of fractional Sobolev norms, \cite{Silvestre,CS20} and \cite[Chapter 2]{FRRO}.  We present our proof for the sphere, but it follows also for two point homogeneous spaces. For another recent and similar result about linear statistics for the DPP given by the projection kernel onto the space of orthogonal polynomials of fixed degree, see \cite{FGY23}.

For the fluctuations of rough statistics of the harmonic ensemble, i.e., for the asymptotic study of the variance of the random variable $n_A$ counting the number of points in $A\subset \Sd$, we prove the following result. 

\begin{thm}\label{variance asymptotic rough}
    Given $N$ points on $\Sd$ drawn from the harmonic ensemble, for any Borel set $A\subset\mathbb{S}^d$
\[
    \lim_{N\to\infty}\frac{\var( n_A )}{N^{1-\frac{1}{d}}\log N}= C_d
      \mathcal H^{d-1}(\partial_* A),
\]
where
\[
C_d=\frac{2^{\frac{1}{d}  
 }\Gamma\left( \frac{d}{2}+1 \right)^2}{d \pi^{3/2} \Gamma(d+1)^{1+\frac{1}{d}}}.
\]
In particular, the limit is finite if and only if $A$ is a Cacciopoli set.
\end{thm}

As before, the asymptotic normality of the centered and normalized counting function follows from Soshnikov result for $d>1$, see \cite{Sos02}, and from Lindeberg central limit theorem in the case $d=1$, see \cite[Theorem 4.9]{Mec19}.
 
An upper bound for the limit above  was proved in \cite{BMOC16} if $A$ is a spherical cap, exploiting a generalization to the sphere of Landau's work about the concentration operator in the space of bandlimited entire functions \cite{Lan67}. 
The approach in \cite{BMOC16} seems difficult to generalize for sets more general than spherical caps. In this paper, we generalize the result in \cite{BMOC16} with a completely different approach, which we present in general terms, since it can be applied to determine the asymptotic behavior of the linear statistics (both smooth and rough) also of other DPPs. 
    
Our main idea is to write, through \eqref{equation_variance},
\begin{equation}\label{eq: general smooth}
\frac{\var\Big( \sum_{i=1}^N f(x_i)\Big)}{C_{smooth}(N)}=
\iint_{\mathbb M\times \mathbb M}\frac{|f(x)-f(y)|^2}{d(x,y)^2} \varrho_{N,smooth}(d(x,y))\,  dV(x)dV(y) ,  
\end{equation}
for $f\in L^2(\mathbb M)$, and
\begin{equation}\label{eq:general rough}
\frac{\var\Big( \sum_{i=1}^N \chi_A(x_i)\Big)}{C_{rough}(N)}=
\iint_{\mathbb M\times \mathbb M}\frac{|\chi_A(x)-\chi_A(y)|}{d(x,y)} \varrho_{N,rough}(d(x,y))\,  dV(x)dV(y),
\end{equation}
for Borel sets $A\subset \mathbb M$, choosing suitable values for $C_{smooth}(N)$ and $C_{rough}(N)$, and then use norm representations similar to those of Bourgain, Brezis and Mironescu \cite{BBM} and of D\'avila \cite{Davila} presented in Section~\ref{sec:function spaces} to prove the asymptotic behavior.  

For the DPPs on the sphere that we consider in this paper, the results on norm representations for general compact manifolds in \cite{Davila-manifold} are not applicable, since they require the radial mollifiers to be strictly decreasing,  a property that the functions $\varrho_N$ appearing in \eqref{eq: general smooth} and \eqref{eq:general rough} for the harmonic and the spherical ensemble do not satisfy, as one can easily check. Even for $\mathbb M=\Sd$, the only result that we know of which does not require the monotonicity restriction, is the analogue of the Bourgain, Brezis and Mironescu norm representation due to Pegon \cite[Proposition 5.7]{pegon}, but it only covers the case $p=2$. In particular, it does not apply to the rough statistics.

This provides us a motivation to prove an analogue of the Bourgain-Brezis-Mironescu-Dávila results \cite[Theorem 2 and 3]{BBM}, \cite[Theorem 1]{Davila} for the unit sphere. Such a result is, although surely expected, not completely straightforward.

\begin{thm}\label{thm:davila sphere}
    Let $\{\varrho_L\}_{L\ge 0}$ 
    be a family of radial mollifiers on $\Sd$ as in \eqref{def_mollifier}. If $1<p<\infty$ and $f \in L^p(\Sd)$, or $p=1$ and $f \in W^{1,1}(\Sd)$, then
    \begin{equation}\label{limit_pegon_p}
        \lim_{L\to \infty}\iint_{\Sd\times\Sd}\frac{|f(x)-f(y)|^p}{d(x,y)^p}\varrho_L(d(x,y))\, d\sigma(x)d\sigma(y)=K_{d,p}\int_{\Sd}|\nabla f|^p.
    \end{equation}
If $f \in L^1(\Sd)$, then
\begin{equation}\label{limit_pegon_BV}
        \lim_{L\to \infty}\iint_{\Sd\times\Sd}\frac{|f(x)-f(y)|}{d(x,y)}\varrho_L(d(x,y))\, d\sigma(x)d\sigma(y)=K_d[f]_{BV}.
    \end{equation}
\end{thm}
As already mentioned, the conclusion of the theorem for $f \in W^{1,2}(\Sd)$ was first obtained by Pegon (under an extra assumption on the mollifiers for $d=1$), while for function in $L^p(\Sd)\setminus W^{1,2}(\Sd)$, $1\leq p<\infty$, it is new.
 Moreover, the authors in \cite{Davila-manifold} mention the fact that the monotonicity assumption in their theorem is probably unnecessary: Theorem~\ref{thm:davila sphere} can be considered a confirmation that their conjecture is true in the case of the unit sphere.

From Theorem~\ref{thm:davila sphere} and \eqref{eq:general rough}, the asymptotic behavior of the rough linear statistics of Theorem~\ref{variance asymptotic rough} follows once we build the right mollifiers.

We now present the results concerning the asymptotic behavior of the linear statistics for the spherical ensemble. For the smooth statistics, we prove the following:

\begin{thm}\label{variance asymptotic spherical smooth}
Let $x_1,\dots, x_N\in \mathbb S^2$ be  points drawn from the spherical ensemble, and $f\in L^2 (\mathbb{S}^2)$. Then,
\[
\lim_{N\to\infty} \var\Big( \sum_{i=1}^N f(x_i)\Big)=\int_{\mathbb S^2}|\nabla f(x)|^2\, d\sigma(x).
\]
In particular, the limit is finite if and only if $f\in H^1 (\mathbb{S}^2)$.
\end{thm}

This result was first proved in \cite{RV07}, see also \cite[Theorem 1.5]{Ber18} for a generalization. Here we prove it as a simple consequence of Theorem~\ref{thm:davila sphere}.

We point out that the same strategy cannot be applied to study the smooth statistics of the harmonic ensemble. Indeed, in that case the functions $L^{-(d-1)}K_L(x,y)^2d(x,y)^2$ appearing in \eqref{eq: general smooth} are not radial mollifiers, and Theorem~\ref{thm:davila sphere} does not apply. This justifies the different approach we use to prove Theorem~\ref{variance asymptotic}, which shows that the right function space to measure the fluctuations of the smooth linear statistics of the harmonic ensemble is not $H^1(\Sd)$, but  $H^{1/2}(\Sd)$.

For the rough statistics of the spherical ensemble, we prove the following.

\begin{thm}\label{variance asymptotic spherical rough}
Given $N$ points on $\mathbb S^2$ drawn from the spherical ensemble, for any Borel set $A\subset\mathbb S^2$, we have
\[
    \lim_{N\to\infty} \frac{\var(n_A)}{\sqrt{N}}={\frac{1}{4\pi\sqrt{\pi}}}\mathcal H^{1}(\partial_* A).
\]
In particular, the limit is finite if and only if $A$ is a Cacciopoli set.
\end{thm}

This limit was obtained in \cite{AZ15} for the special case of spherical caps. 
Observe that our constant matches the one in \cite{AZ15} by using L\'evy's isoperimetric inequality, which says that for any domain $A\subset \mathbb S^2$ bounded by a closed curve
\[
\mathcal{H}^1 (\partial A)\ge 4\pi\sqrt{\sigma(A)(1-\sigma(A))},
\]
with equality when $A\subset \mathbb S^2$ is a spherical cap. Here we are able to prove Theorem~\ref{variance asymptotic spherical rough} for all Caccioppoli sets, using again, as for the rough statistics of the harmonic ensemble, Theorem~\ref{thm:davila sphere}.

Let us now present our results regarding point processes in the Euclidean space. We start from the infinite Ginibre process. The asymptotics now involves the increment of the intensity of the number of points. We are able to show that one can write the variance of the smooth and the rough linear statistics as in \eqref{eq: general smooth} and \eqref{eq:general rough}, respectively, in such a way that the functions $\varrho_N$ appearing on the right are radial mollifiers on $\mathbb C$. Then, using the norm representation results for $\Omega=\mathbb R^d$ by Brezis \cite{brezis}, we deduce the following results.

\begin{thm}\label{Ginibre-smooth}
Let $f \in L^2(\mathbb C)$, and let $\Xi_L$ be the determinantal point process with intensity $L/\pi$ induced by the infinite Ginibre kernel \eqref{kernel_ginibre}. Then,
 \[
  \lim_{L\to\infty} \var\Bigl(\sum_{x\in \Xi_L} f(x)\Bigr) =\frac{1}{4\pi}\int_{\mathbb C}|\nabla f(z)|^2\, dz.
 \]
  In particular, the limit is finite if and only if $f\in H^1(\mathbb C)$.
\end{thm}

\begin{thm}\label{Ginibre-rough}
Let $A\subset \mathbb C$ be a Borel set, and let $\Xi_L$ be the determinantal point process with intensity $L/\pi$ induced by the infinite Ginibre kernel \eqref{kernel_ginibre}. Then,
 \[
  \lim_{L\to\infty} \frac {\var( n_A )}{\sqrt{L}} = \frac{1}{2\pi\sqrt{\pi}} \mathcal H^1(\partial_* A ).
 \]
  In particular, the limit is finite if and only if $A$ is a Cacciopoli set.
\end{thm}
To obtain analogous results for the finite Ginibre ensemble, the norm representation results available in the literature are not sufficient, because the kernel, and hence the associated mollifiers, are not radial. This leads us to introduce the notion of $p$-asymptotically radial mollifiers (see Section \ref{sec: norm representation euclidean space} for the definition) and prove the following norm representation result.

\begin{thm}\label{thm:asymptotically radial mollifiers}
Let $1\leq p<\infty$ and $f$ be a function with compact support $K$. Let $\varrho_n:  \R^d\times \R^d\to \R$ be asymptotically radial $p$-mollifiers in an open neighborhood $\Omega$ of $K$.  If $1<p<\infty$ and $f \in L^p(\R^d)$, or $p=1$ and $f \in W^{1,1}(\R^d)$, then
\[
\lim_{n\to\infty} \iint_{\R^d\times \R^d}\frac{|f(x)-f(y)|^p}{|x-y|^p}\varrho_n(x,y)\, dxdy = K_{d,p}\int_{\R^d}|\nabla f|^p.
\]
If $f \in L^1(\R^d)$, then
\[
\lim_{n\to\infty} \iint_{\R^d\times \R^d}\frac{|f(x)-f(y)|}{|x-y|}\varrho_n(x,y)\, dxdy = K_{d}[f]_{BV}.
\]
\end{thm}

We are able to prove that the functions $\varrho_N$ appearing writing the linear statistics of the finite Ginibre ensemble as in \eqref{eq: general smooth} and \eqref{eq:general rough} are $p$-asymptotically radial mollifiers, for $p=2$ and $p=1$, respectively. Then, an application of Theorem~\ref{thm:asymptotically radial mollifiers} proves the following asymptotic behaviors.

\begin{thm}\label{finite Ginibre-smooth}
Let $f \in L^2(\mathbb C)$, with compact support contained in the open disc $\mathbb D$, and $x_1,\dots, x_N\in \mathbb C$ be  points drawn from the finite Ginibre ensemble. Then,
 \[
 \lim_{N\to\infty} \var\Big( \sum_{i=1}^N f(x_i)\Big) =\frac{1}{4\pi}\int_{\mathbb C}|\nabla f(z)|^2\, dz.
 \]
  In particular, the limit is finite if and only if $f\in H^1(\mathbb C)$.
\end{thm}

\begin{thm}\label{finite Ginibre-rough}
Let $A$ be a Borel set such that $\overline{A}\subset \mathbb D$, and $x_1,\dots, x_N\in \mathbb C$ be  points drawn from the finite Ginibre ensemble. Then,
 \[
  \lim_{N\to\infty} \frac {\var( n_A )}{\sqrt{N}} = \frac{1}{2\pi\sqrt{\pi}} \mathcal H^1(\partial_* A ).
 \]
  In particular, the limit is finite if and only if $A$ is a Cacciopoli set.
\end{thm}

Observe that the results in Theorems \ref{Ginibre-smooth} and \ref{finite Ginibre-smooth} matches that obtained, by means of different methods, in \cite[Theorem 1.1]{virag} for the finite Ginibre ensemble. In that case the support of the function $f$ is not requested to lie inside the disc, and also the $H^{1/2}$ norm of the function on the unit circle contributes, together with the $H^1$ norm in the open disc, to determine the asymptotic value of the variance. In our case, there is no boundary contribution and the $H^1$ norm of the function summarizes all the information. The asymptotic behavior of the rough statistics of the finite dimensional Ginibre ensemble was already carried out in \cite{Lin23}.

The smooth linear statistics of the Bessel point process on $\mathbb R^d$ are particularly simple to compute, and follow essentially from the Riemann-Lebesgue lemma. We include the proof for completeness.

\begin{prop}\label{rough_bessel}
 Let $f\in L^2(\R^d)$ and let $\Xi_L$ be a determinantal point process with intensity $L^d$ induced by the Bessel kernel \eqref{kernel_bessel}. Then,
 \begin{align*}
  \lim_{L\to\infty} & \frac 1{L^{d-1}}\var\Bigl(\sum_{x\in \Xi_L} f(x)\Bigr) 
  = \frac{ 2^{d-1}\Gamma\left( \frac{d}{2}+1\right)^2}{\pi^{d+2}} 
  \iint_{\R^d\times\R^d} \frac{|f(x)-f(y)|^2}{|x-y|^{d+1}}\, dxdy.
 \end{align*}
In particular, the limit is finite if and only if $f\in H^{1/2}(\R^d)$.
\end{prop}

For the rough case, although the kernel is radial,  nevertheless we cannot apply the norm representation result for BV functions by Brezis, because the candidate mollifiers turn out not being integrable. Once again, however, we can apply Theorem~\ref{thm:asymptotically radial mollifiers}, since the functions $\varrho_L$ arising from writing the rough linear statistics as in \eqref{eq:general rough} fulfill the weaker integrability assumptions of the asymptotically radial mollifiers. Then, we obtain the following result.

\begin{thm}\label{Bessel-rough}
Let $A\subset \R^d$ be a Borel set, and let $\Xi_L$ be the determinantal point process with intensity $L^d$ induced by the Bessel kernel \eqref{kernel_bessel}. Then,
 \[
  \lim_{L\to\infty} \frac {\var( n_A )}{L^{d-1}\log L}=C_d\mathcal H^{d-1}(\partial_* A),
 \]
 where
 \[
 C_d=\frac{2^d\Gamma(\frac{d}{2}+1)^2}{\pi^{(d+4)/2}\Gamma((d+1)/2)}.
 \]
  In particular, the limit is finite if and only if $A$ is a Cacciopoli set.
\end{thm}

\section*{Proofs}

\subsection{Harmonic ensemble: discrepancy and smooth statistics} We begin by proving the result on the asymptotic equidistribution of the harmonic ensemble.

\begin{proof}[Proof of Proposition~\ref{equidistribution}]
Observe that for a fixed geodesic ball $D\subset \mathbb M$
the function
$\mu \mapsto \mu(D)$ is 1-Lipschitz with respect to the total variation distance as a function on finite counting measures.
Therefore, we can apply the concentration result of Pemantle and Peres for determinantal point processes in \cite[Theorem 3.5]{PP14}, that states that 1-Lipschitz functionals on determinantal process are concentrated around the mean. More precisely,
there exists $A>0$ such that
for all $M>0$
\[
\mathbb P\left( |n_D-\mathbb E (n_D)| >\sqrt{M N\log N} \right)\le
A \exp\left(  
-\frac{M N \log N}{16\sqrt{ M N\log N}+2N)}
\right)\le
\frac{A}{N^{M/32}},
\]
for $N$ big enough depending on $M$. 

Now, for $N$ fixed, there exists a family 
$\mathcal A_N$ 
of cardinality of order $N^{d+1}$ containing geodesic balls such that, for 
any geodesic ball $D\subset \mathbb M$, there are
$D_1,D_2 \in \mathcal A_N$ with
$D_1\subset D \subset D_2$ and $V(D_2\setminus D_1)\lesssim \frac{1}{N}$. Indeed, consider all geodesic balls of radius $1/N$. By Vitali's covering lemma, there exists a countable sub-collection of disjoint balls
such that $\mathbb M$ is contained in the union of a fixed dilate of the balls. From volume considerations, the sub-collection has cardinality of order $N^d$. Let now $\mathcal A_N$ be the balls in the sub-collection and all the balls with the same center and  radius $j/N$ or all $j=0,\dots ,N.$

Observe that, for any $D\subset \mathbb M$ and corresponding
$D_1,D_2 \in \mathcal A_N$
\[
n_{D_1}-\mathbb E(n_{D_1})+\mathbb E(n_{D_1})-\mathbb E(n_{D_2})\le 
n_D-\mathbb E(n_D)\le 
n_{D_2}-\mathbb E(n_{D_2})+\mathbb E(n_{D_2})-\mathbb E(n_{D_1}).
\]
Let $K_L(x,y)$ be the kernel of the harmonic ensemble on $\mathbb M$ with $N=k_L$. From the
local Weyl's law \cite{Hor68}
$K_L(x,x)=N+\bigO{N^{1-\frac{1}{d}}}$, we get that for any geodesic ball $D\subset \mathbb M$
\[
|n_D-\mathbb E(n_D)|\le 1+\mathcal O(N^{-1/d})+
\max_{D_i\in \mathcal A_N}|n_{D_i}-\mathbb E(n_{D_i})|.
\]
By the union bound we get for $X_N\subset \mathbb M$ drawn from the harmonic ensemble 
\begin{align*}
&\mathbb P\left( \sup_{D\subset \mathbb M} |n_D-\mathbb E(n_D)|  >\sqrt{M N\log N} \right)\\
&\le \sum_{D \in \mathcal A_N}\mathbb P\left( |n_D-\mathbb E(n_D)| >\sqrt{(M+1) N\log N} \right)\lesssim  \frac{1}{N^2},
\end{align*}
for $M=32(d+3)$ and all $N$ big enough. 
From this bound and Borel-Cantelli's lemma it follows
\[
\mathbb P\left( \limsup_{N\to +\infty} \Big\{
\sup_{D\subset \mathbb M} |n_D-\mathbb E(n_D)| >\sqrt{M N\log N}\Big\} \right)=0,
\]
and therefore, there exist $N_0$ such that for $N>N_0$
\[
\sup_{D\subset M}\left| \frac{n_D}{N}-\frac{\mathbb{E}(n_D)}{N}\right|\le C \sqrt{\frac{\log N}{N}},
\]
with probability 1. 
Finally, we apply again the local Weyl's law to get
\[
\frac{\mathbb E(n_D)}{N} = \frac{1}{N}\int_D K_L(x,x)\, dV (x)\rightarrow V(D),\quad N\to +\infty.
\]
\end{proof}

In order to prove Theorem~\ref{variance asymptotic} regarding the asymptotic behavior of the variance for the smooth statistics of the harmonic ensemble, we will need the following result, which is an adaptation of \cite[Lemma 4.2]{Silvestre} to Riemannian manifolds.

\begin{prop}\label{silvestre}
Let $\mathbb{M}$ be a $d$-dimensional Riemannian manifold with injectivity radius $r_\mathbb{M}>0$. Denote by $B(x,r)$ the ball of radius $r$ centered at $x\in \mathbb{M}$ with respect to the geodesic distance $d(\cdot,\cdot)$, and assume that there exists a constant $c_\mathbb M\leq r_\mathbb{M}/2$ such that $V(B(x,r))\gtrsim r^d$ for $r\leq c_{\mathbb M}$. Let $K:\mathbb{M}\times \mathbb{M}\to \mathbb{R}$ be a symmetric kernel such that for any $r\in (0,r_\mathbb{M}/2)$ and any $x\in \mathbb{M}$,
\[    \int_{B(x,2r)\setminus B(x,r)}|K(x,y)|\, dV(y)\leq \Lambda r^{-2s},
\]
for some $\Lambda>0$ and $s\in(0,1)$. Then, there exists a constant $C=C(\Lambda,d,s)$ such that for any $g\in H^s(\mathbb{M})$,
\begin{align*}
   &\iint_{d(x,y)<2c_\mathbb{M}} |g(x)-g(y)|^2K(x,y)\, dV(x)dV(y)  \\ 
   &\leq C \iint_{d(x,y)<3c_\mathbb{M}/2}\frac{|g(x)-g(y)|^2}{d(x,y)^{d+2s}}\, dV(x)dV(y) \leq C [g]_{s,2}^2.
\end{align*}
\end{prop}
\begin{proof}
We may assume without loss of generality that the kernel is non-negative. For any $r>0$ let
\[
    I(r):=\iint_{r<d(x,y)<2r} |g(x)-g(y)|^2K(x,y)\, dV(x)dV(y),
\]
so that
\begin{equation}\label{dyadic sum}
   \iint_{d(x,y)<2c_\mathbb{M}} |g(x)-g(y)|^2K(x,y)\, dV(x)dV(y)=\sum_{k=-\infty}^0 I(2^kc_\mathbb{M} ).
\end{equation}
Given $x,y\in \mathbb{M}$ with $r<d(x,y)<2r\leq 2c_\mathbb{M}$, write $m(x,y)$ for their geodesic mean point, i.e., the unique point lying on the unique geodesic from $x$ to $y$ with $d(x,m(x,y))=d(y,m(x,y))$. Then,
\[
\begin{split}
    r^d|g(x)-g(y)|^2&\lesssim \int_{B(m(x, y),d(x,y)/4)}|g(x)-g(y)|^2\, dV(w)\\
    &\leq 2\int_{B(m(x, y),d(x, y)/4)}(|g(x)-g(w)|^2 +|g(w)-g(y)|^2)\, dV(w)\\
    &\leq 2\int_{r/4<d(x,w)<3r/2}|g(x)-g(w)|^2\, dV(w) \\&+2\int_{r/4<d(y,w)<3r/2}|g(y)-g(w)|^2\, dV(w).
\end{split}  
\]
Hence, for any $r\in (0,c_\mathbb{M})$ we have
\[
\begin{split}
    I(r)&\lesssim \frac{1}{r^d}\iint_{r<d(x,y)<2r}\int_{r/4<d(x,w)<3r/2}|g(x)-g(w)|^2dV(w) K(x,y)\, dV(x)dV(y)\\
    &\leq \frac{1}{r^d}\int_{\mathbb{M}}\Big(\int_{r/4<d(x,w)<3r/2}|g(x)-g(w)|^2dV(w)\Big)\Big(\int_{B(x,2r)\setminus B(x,r)} K(x,y)\,  dV(y)\Big)\, dV(x)\\
    &\leq \frac{\Lambda}{r^{d+2s}}\int_{\mathbb{M}}\int_{r/4\leq d(x,w)<3r/2}|g(x)-g(w)|^2\, dV(w)dV(x)\\
    &\leq \Big(\frac{3}{2}\Big)^{d+2s}\Lambda \int_{\mathbb{M}}\int_{r/4\leq d(x,w)<3r/2}\frac{|g(x)-g(w)|^2}{d(x,w)^{d+2s}}\, dV(w)dV(x)=:Q(r).
\end{split}
\]
Inserting this in \eqref{dyadic sum} we obtain
\[
\begin{split}
        &\iint_{d(x,y)<2c_\mathbb{M}} |g(x)-g(y)|^2K(x,y)\, dV(x)dV(y)\lesssim\sum_{k=-\infty}^0 Q(2^k c_\mathbb{M} )\\
        &= \iint_{d(x,y)<3c_\mathbb{M}/2}\frac{|g(x)-g(y)|^2}{d(x,y)^{d+2s}}\, dV(x)dV(y).
\end{split}
\]
\end{proof}

Observe that Proposition \ref{silvestre} applies in particular to all compact manifolds, since they are Ahlfors regular, and also to all (not necessarily compact) manifolds with sectional curvatures uniformly bounded above by $\kappa\in \mathbb R$. Indeed, as a consequence of the Günter-Bishop comparison theorem, in such a case $V(B(x,r))\gtrsim r^d$ for all $x\in \mathbb M$, and all $r\leq r_{\mathbb M}$, if $\kappa\leq 0$, or all $r\leq\min\{r_{\mathbb M},\pi/\sqrt{\kappa}\}$, if $\kappa>0$. In the proof of the next theorem, we will apply the result to $\mathbb M=\Sd$.

\begin{proof}[Proof of Theorem~\ref{variance asymptotic}]

Let $K_L(x,y)=K_L(d(x,y))$ be the kernel of the harmonic ensemble on $\Sd$ with $N$ points, where $N=\pi_L$ and $\pi_L$ is defined as in \eqref{eq:piL}. Then,
\[
    \lim_{N\to\infty} \frac{\var( n_A )}{N^{1-\frac{1}{d}}}=\left( \frac{\Gamma(d+1)}{2} \right)^{1-\frac{1}{d}}\lim_{L\to \infty}\frac{\var( n_A )}{L^{d-1}}.
\]
Then, by \eqref{equation_variance}, what we want to prove is
\[    \lim_{L\to\infty}\frac{1}{L^{d-1}}\iint |g(x)-g(y)|^2 K_L(x,y)^2 \, d\sigma(x)d\sigma(y)=\vertiii{g}_{1/2}^2,
\]
where $\vertiii{\cdot}_{1/2}^2$ is an equivalent (up to the dimension $d$) seminorm to $ [\,\cdot\,]_{1/2}^2$.

The kernel $K_L$ satisfies the following global pointwise bound, due to Hörmander, see \cite{HormanderRiesz},
\begin{equation}\label{eq:hormander}
   K_L(\cos\theta)\leq C\frac{L^d}{1+L\theta}.
\end{equation}
Localizing $\theta$ in the interval $[\varepsilon/L,\pi-\varepsilon/L]$, where $\varepsilon$ is any fixed number in $(0,\pi/2)$, also the following more precise estimate by Szegö (cf. for example \cite[Theorem 8.21.13]{szego}) is available,
\begin{equation}\label{szego interval}
    P_L^{(d/2,d/2-1)}(\cos{\theta})=L^{-1/2}k(\theta)h_L(\theta), \quad \varepsilon/L\leq\theta\leq\pi-\varepsilon/L,
\end{equation}
where
\[    k(\theta)=\sqrt{\pi^{-1}(\sin{\theta/2})^{-(d+1)}(\cos{\theta/2})^{-(d-1)}},
\]
and 
\[
h_L(\theta)=\cos\Big((L+d/2)\theta-\frac{\pi(d+1)}{4}\Big)+\frac{\mathcal O(1)}{L\sin\theta}.
\]
The bound for the error term, $\mathcal O(1)\leq C_\varepsilon$, is uniform in the interval $[\varepsilon/L,\pi-\varepsilon/L]$.

Fix $\varepsilon\in (0,1)$. For the near to the diagonal part of the above integral we can exploit the pointwise bound \eqref{eq:hormander} to get
\begin{equation}\label{omega 1}
\begin{split}
   &\frac{1}{L^{d-1}}\iint_{d(x,y)<\varepsilon/L} |g(x)-g(y)|^2 K_L(x,y)^2 \, d\sigma(x)d\sigma(y)\\
   &\leq  CL^{d-1}\iint_{d(x,y)<\varepsilon/L} \frac{|g(x)-g(y)|^2}{d(x,y)^2}\, d\sigma(x)d\sigma(y)\\
   &\leq C\varepsilon^{d-1} \iint_{d(x,y)<\varepsilon/L} \frac{|g(x)-g(y)|^2}{d(x,y)^{d+1}}\, d\sigma(x)d\sigma(y)\leq C\varepsilon^{d-1} [g]_{1/2}^2.
\end{split}    
\end{equation}
On the other hand, again by \eqref{eq:hormander},
\begin{equation}\label{omega 3}
\begin{split}
       &\frac{1}{L^{d-1}}\iint_{\pi-\varepsilon/L< d(x,y)\leq \pi} |g(x)-g(y)|^2 K_L(x,y)^2 \, d\sigma(x)d\sigma(y)\\
       &\leq \frac{2C}{\pi}L^{d-1} \iint_{\pi-\varepsilon/L< d(x,y)\leq \pi} |g(x)-g(y)|^2\, d\sigma(x)d\sigma(y)\\
       &\leq \frac{2C}{\pi}L^{d-1}4\iint_{\pi-\varepsilon/L< d(x,y)\leq \pi} |g(x)|^2d\sigma(x)\, d\sigma(y)\\
       &=\frac{8C}{\pi}L^{d-1}\sigma(\{\pi-\varepsilon/L< d(x,y)\leq \pi\})\int |g(x)|^2\, d\sigma(x)\\
       &\approx \Vert g\Vert_{L^2(\Sd)}^2\varepsilon^dL^{-1}\longrightarrow 0, \quad L\to \infty.
\end{split}
\end{equation}
We now turn our attention to the remaining part of the integral, that is, that on the region $\Omega_\varepsilon(L)=\{(x,y)\in \mathbb{S}^d\times\mathbb{S}^d: \ \varepsilon/L\leq d(x, y)\leq\pi-\varepsilon/L\}$. Recall that on this region, by \eqref{szego interval}, we have
\[   K_L(x,y)^2=\frac{c_{d,L}^2}{\sigma(\mathbb{S}^d)^2}L^{-1}k(x,y)^2h_L(x,y)^2.
\]
The key observation here is that the kernel $k(x,y)^2$ satisfies
\begin{equation}\label{consequence silvestre}
    \iint_{\Sd\times\Sd} |g(x)-g(y)|^2 k(x,y)^2\, d\sigma(x)d\sigma(y)=2\vertiii g_{1/2}^2,
\end{equation}
where $\vertiii{\cdot}_{1/2}^2$ is an equivalent (up to the dimension $d$) seminorm to $[\,\cdot\,]_{1/2}^2$. Indeed, on the one hand,
\[
    k(x,y)^2\geq \sin(d(x,y)/2)^{-(d+1)}\geq d(x,y)^{-(d+1)},
\]
which gives
\[
\iint_{\Sd\times\Sd}  |g(x)-g(y)|^2k(x,y)^2\, d\sigma(x)d\sigma(y)\geq [g]_{1/2}^2.
\]
On the other hand, for any $x\in\mathbb{S}^d$ and any $r>0$,
\[
\begin{split} \int_{r\leq d(x,y)\leq \pi}k(x,y)^2\,  dy&= \frac{1}{\pi}\int_{r\leq d(x,y)\leq \pi}\frac{d\sigma(y)}{(\sin{ d(x,y)/2})^{d+1}(\cos{ d(x,y)/2})^{d-1}}\\
&=\frac{\omega_{d-1}}{\omega_d\pi}\int_r^{\pi}\frac{(\sin\rho)^{d-1}\, d\rho}{(\sin{\rho/2})^{d+1}(\cos{\rho/2})^{d-1}}\\
&=\frac{\omega_{d-1}}{\omega_d2^{d-1}\pi}\int_r^{\pi}\frac{d\rho}{(\sin{\rho/2})^2}\\
&=\frac{\omega_{d-1}}{\omega_d2^{d-2}\pi}
 \cot (r/2)\leq \frac{\omega_{d-1}}{\omega_d2^{d-3}\pi r}.
\end{split}
\]
Hence, Proposition~\ref{silvestre} gives,
\begin{equation}\label{consequence silvestre upper bound} \iint_{\Sd\times\Sd}  |g(x)-g(y)|^2k(x,y)^2\, d\sigma(x)d\sigma(y)\leq C [g]_{1/2}^2,
\end{equation}
which proves \eqref{consequence silvestre}.

Now we observe that the kernel $h_L(x,y)^2$ satisfies
\begin{equation}\label{h_L}
    h_L(\theta)^2=\frac{1}{2}+\frac{1}{2}a_{d,L}(\theta)+b_L(\theta)\bigO{1},
\end{equation}
where
\[    a_{d,L}(\theta):=\cos\Big((2L+d)\theta-\frac{\pi(d+1)}{2}\Big)\in\{\pm\cos{((2L+d)\theta)},\pm\sin{((2L+d)\theta)}\},
\]
and
\[
    b_L(\theta)=\frac{2}{L\sin\theta}+\frac{1}{L^2\sin^2\theta}.
\]
Observe that $b_L\to 0$ as $L\to \infty$, pointwise. Moreover,
\[
    |b_L(\theta)|\leq \frac{3}{\sin^2(\theta)}=\frac{3}{2\sin^2(\theta/2)\cos^2(\theta/2)},
\]
so that, $|k( x,y)^2b_L( x,y)\mathcal O(1)|\lesssim C_\varepsilon k_{d+2}( x,y)^2$. But by \eqref{consequence silvestre upper bound},
\[
     \iint |g(x)-g(y)|^2 C_\varepsilon k^{d+2}( x,y)^2\, d\sigma(x)d\sigma(y)\lesssim C_\varepsilon[g]_{1/2}^2.
\]
Therefore, by the dominated convergence theorem,
\begin{equation}\label{b part}
  \lim_{L\to\infty}\iint |g(x)-g(y)|^2k( x,y)^2b_L( x,y)\bigO{1}\, d\sigma(x)d\sigma(y)=0. 
\end{equation}
Now, consider the map $\pi: \Sd\times\Sd \to [0,\pi]$ given by $\pi(x,y)=d(x,y)$. By the disintegration theorem, there exists a family of probability measures $\{\mu_\theta\}_{\theta\in[0,\pi]}$ on $\Sd\times\Sd$ such that $\mu_\theta$ is supported on $\pi^{-1}(\theta)$ and
\begin{align*}
    &\iint_{\Sd\times\Sd}  |g(x)-g(y)|^2k(d(x,y))^2\, d\sigma(x)d\sigma(y)\\
    & = \int_{[0,\pi]}\iint_{ d(x,y)=\theta}|g(x)-g(y)|^2k(\theta)^2\, d\mu_\theta(x,y)d\nu(\theta),
\end{align*}
where $\nu=\pi_*(\sigma\otimes \sigma)$, i.e., $d\nu(\theta)=B_d\sin^{d-1}(\theta)d\theta$. Hence, by \eqref{consequence silvestre upper bound}, the function
\[   f(\theta):=B_d\iint_{ d(x,y)=\theta}|g(x)-g(y)|^2k(\theta)^2\sin^{d-1}(\theta)d\mu_\theta(x,y),
\]
belongs to $L^1([0,\pi],d\theta)$. 
Therefore, again by the disintegration theorem
\begin{equation}\label{a part}
\iint_{\Sd\times\Sd} |g(x)-g(y)|^2k(d(x,y))^2a_{d,L}( d(x,y))\, d\sigma(x)d\sigma(y)=\int_{-\infty}^{\infty} \chi_{[0,\pi]}(\theta) f(\theta)a_{d,L}(\theta)\, d\theta,
\end{equation}
which tends to 0 as $L\to \infty$ by the Riemann-Lebesgue lemma.

Recalling that $c_{d,L}^2L^{-1}\sim L^{d-1}$ as $L\to\infty$, and using first \eqref{b part} and \eqref{a part} and then \eqref{consequence silvestre}, we get
\begin{equation}\label{omega 2}
\begin{split}
&\lim_{L\to\infty}\frac{1}{L^{d-1}}\iint_{\Omega_\varepsilon(L)} |g(x)-g(y)|^2 K_L(x,y)^2\,  d\sigma(x)d\sigma(y)\\
&= \lim_{L\to\infty}\iint_{\Omega_\varepsilon(L)} |g(x)-g(y)|^2 k(x,y)^2h_L(x,y)^2\,  d\sigma(x)d\sigma(y)\\
&=\lim_{L\to\infty}\frac{1}{2}\iint_{\Omega_\varepsilon(L)} |g(x)-g(y)|^2 k(x,y)^2\,  d\sigma(x)d\sigma(y)=\vertiii g_{1/2}^2.
\end{split}
\end{equation}
This immediately gives
\[
    \lim_{L\to\infty}\frac{1}{L^{d-1}}\iint |g(x)-g(y)|^2 K_L(x,y)^2 d\sigma(x)d\sigma(y)\geq \vertiii g_{1/2}^2.
\]
The reverse bound is obtained observing that
\[
    \lim_{L\to\infty}\frac{1}{L^{d-1}}\iint |g(x)-g(y)|^2 K_L(x,y)^2 d\sigma(x)d\sigma(y)\leq C\varepsilon^{d-1} [g]_{1/2}^2+\vertiii g_{1/2}^2,
\]
which follows by \eqref{omega 1}, \eqref{omega 3} and \eqref{omega 2},
and then letting $\varepsilon\to 0$.

\end{proof}

\subsection{Norm representations in \texorpdfstring{$\Sd$}{the sphere}}

For convenience, we set
\[
    I_{d,p}(\varrho, f):=\iint_{\Sd\times\Sd}\frac{|f(x)-f(y)|^p}{d(x,y)^p}\varrho(d(x,y))\,  d\sigma(x)d\sigma(y),
\]
and use this notation throughout this section. To prove Theorem~\ref{thm:davila sphere}, we first establish an upper bound for $I_{d,p}(\varrho, f)$ which extends \cite[Proposition A.1]{pegon}.

\begin{prop}\label{modified pegon}
    Let $d\geq 1$ and $\varrho: (0,\pi)\to\mathbb{R^+}$ a non-negative measurable function, and set
        \[
          Q(\varrho)=\frac{\omega_{d-1}}{\omega_{d}}\int_0^\pi \varrho(r)\sin^{d-1}(r)\, dr.
    \]
If $f \in W^{1,p}(\mathbb S^d)$, then for any $1\leq p<\infty$,
    \begin{equation}\label{pointwise bound pegon}       I_{d,p}(\varrho, f)\leq K_{d,p}Q(\varrho)\int_{\Sd} |\nabla f(z)|^p\, d\sigma(z).
    \end{equation}
Moreover, if $f\in BV(\Sd)$,
     \[
    I_{d,1}(\varrho, f)\leq K_{d}Q(\varrho)[f]_{BV}.
    \]
\end{prop}

\begin{proof}
Assume $f\in \mathcal{C}^1(\Sd)$. For $x,y\in \Sd$, let $\gamma_{yx}:[0,1]\to\Sd$ be the constant speed parametrization of a geodesic on the sphere such that $\gamma_{yx}(0)=y$, $\gamma_{yx}(1)=x$, and $| \gamma'_{yx}(t)|=d(x,y)$, for all $t\in[0,1]$. Observe that the geodesic is unique outside the negligible set of points $(x,y)$ in $\Sd\times\Sd$ for which $x=\pm y$. Set $e_{xy}=\gamma'_{yx}(t)/| \gamma'_{yx}(t)|$. Then, since
\[
    f(x)-f(y)=\int_0^1 \nabla f(\gamma_{yx}(t))\cdot \gamma'_{yx}(t) dt,
\]
we can write
\begin{equation}\label{pegon formula}
   I_{d,p}(\varrho, f)=\iint_{\Sd\times\Sd}\Big|\int_0^1 \nabla f(\gamma_{yx}(t))\cdot e_{xy}(t)dt\Big|^p \varrho(d(x,y))\,  d\sigma(x)d\sigma(y).
\end{equation}
Then, by Jensen's inequality and Fubini's theorem,
\[
\begin{split}
   I_{d,p}(\varrho, f)&\leq \iint_{\Sd\times\Sd}\int_0^1|\nabla f(\gamma_{yx}(t))\cdot e_{xy}(t)|^p\, dt\varrho(d(x,y))\, d\sigma(x)\, d\sigma(y)\\
    &=\int_{\Sd}\int_0^1\int_{\Sd}|\nabla f(\gamma_{yx}(t))\cdot e_{xy}(t)|^p \varrho(d(x,y))\, d\sigma(x)\, dt\, d\sigma(y).
\end{split}
\]

The next thing we want to do is to perform a change of variable in the inner integral from $x$ to the variable $z=\gamma_{yx}(t)$. Observe that $\gamma'_{yx}(t)$ is a constant multiple of the vector $\langle y,\gamma_{yx}(t)\rangle\gamma_{yx}(t)-y$, so that with respect to the new variable we have
\[
    e_{xy}(t)=\frac{\langle y,z\rangle z-y}{| \langle y,z\rangle z-y |}.
\]
We now make the change of variable explicit. Let $\Phi_y:\mathbb{S}^{d-1}\times (0,\pi)\to \Sd\setminus\{y\}$ be the spherical coordinates on the sphere having $y$ as the North Pole, i.e., if $p=\Phi_y(\xi,r)$, then $r=d(p,y)$ and $\xi$ is
the unique point on the equator of the sphere such that $p,\xi$ and $y$ lie on the same geodesic. If $x=\Phi_y(\xi,d(x,y))$ in these coordinates, then it is clear that $z=\Phi_y(\xi,d(y,z))$, so that
\begin{equation}\label{eq: change of variable}
x=\Phi_y(\xi,d(y,z)/t)=\Phi_y\circ \tau_t (\xi,d(y,z))= \Phi_y\circ \tau_t \circ \Phi_y^{-1}(z)=:T_{y,t}(z),
\end{equation}
where $\tau_t:\mathbb{S}^{d-1}\times \mathbb{R}_+\to \mathbb{S}^{d-1}\times \mathbb{R}_+$ is the map given by
\[
    \tau_t=\begin{bmatrix}
I & \mathbf{0} \\
\mathbf{0}^T & 1/t
\end{bmatrix}.
\]
Then, we have,
\begin{equation}\label{first change of variable}
\begin{split}
    &I_{d,p}(\varrho,f)\leq \int_{\Sd}\int_0^1\int_{\Sd}|\nabla f(\gamma_{yx}(t))\cdot e_{xy}(t)|^p\varrho(d(x,y))\, d\sigma(x)dtd\sigma(y)\\
    &=\int_{\Sd}\int_0^1\int_{d(y,z)\leq t\pi}\Big|\nabla f(z)\cdot \frac{\langle y,z\rangle z-y}{| \langle y,z\rangle z-y |}\Big|^p\varrho(d(y,z)/t)|\det \mathcal{J}T_{y,t}(z)|\, d\sigma(z)dtd\sigma(y)\\
   &= \int_{\Sd}\int_0^1\int_{d(y,z)\leq t\pi}\Big|\nabla f(z)\cdot \frac{\langle y,z\rangle z-y}{| \langle y,z\rangle z-y |}\Big|^p\varrho(d(y,z)/t)|\det \mathcal{J}T_{y,t}(z)|\, d\sigma(y)dtd\sigma(z),
\end{split}
\end{equation}
where in the last line we simply applied Fubini's theorem. Next step is to apply the change of variable $y=\Phi_z(\xi,r)$ in the inner integral. It is easily seen that under this change
\[
    \frac{\langle y,z\rangle z-y}{| \langle y,z\rangle z-y |}=\Phi_z(-\xi,\pi/2).
\]
Moreover, by \eqref{eq: change of variable}, we have
\[
    |\det \mathcal{J}T_{y,t}(z)|=\frac{1}{t}\Big|\frac{\det\mathcal{J}\Phi_y(\xi,d(y,z)/t)}{\det\mathcal{J}\Phi_y(\xi,d(y,z))}\Big|=\frac{1}{t}\Big|\frac{\det\mathcal{J}\Phi_y(\xi,r/t)}{\det\mathcal{J}\Phi_y(\xi,r)}\Big|=\frac{1}{t}\Big|\frac{\det\mathcal{J}\Phi_z(\xi,r/t)}{\det\mathcal{J}\Phi_z(\xi,r)}\Big|,
\]
where in the last equality we used the fact that the Jacobian determinant of the spherical change of variables only depends on the distance of each point to the North Pole, and not on the choice of the pole itself. Hence, since $\omega_{d} d\sigma(y)=\omega_{d-1}|\det\mathcal{J}\Phi_z(\xi,r)|d\sigma(\xi) dr$, we obtain 
\[
   |\det \mathcal{J}T_{y,t}(z)|d\sigma(y)=\frac{\omega_{d-1}}{\omega_{d}}\frac{1}{t}|\det\mathcal{J}\Phi_y(\xi,r/t)|d\sigma(\xi) dr = \frac{\omega_{d-1}}{\omega_{d}}\frac{1}{t}\sin^{d-1} (r/t) d\sigma(\xi) dr.
\]
By rewriting the last line in \eqref{first change of variable} as
\[
\begin{split}
    &\frac{\omega_{d-1}}{\omega_{d}}\int_{\Sd}\int_0^1 \int_0^{t\pi}\frac{1}{t}\varrho(r/t)\sin^{d-1} (r/t)\int_{\mathbb{S}^{d-1}}\Big| \nabla f(z)\cdot \Phi_z(-\xi,\pi/2)\Big|^p \, d\sigma(\xi) drdtd\sigma(z)\\
    &=\frac{\omega_{d-1}}{\omega_{d}}K_{d,p} \int_{\Sd} |\nabla f(z)|^p \, d\sigma(z) \int_0^1 \int_0^{t\pi}\frac{1}{t} \varrho(r/t)\sin^{d-1} (r/t)  \, drdt\\
    &=\frac{\omega_{d-1}}{\omega_{d}}K_{d,p} \int_{\Sd} |\nabla f(z)|^p \, d\sigma(z) \int_0^1 \int_0^{\pi} \varrho(r)\sin^{d-1} (r)\,   dr dt = K_{d,p} Q(\varrho)\int_{\Sd} |\nabla f(z)|^p\,  d\sigma(z),
\end{split}
\]
which proves \eqref{pointwise bound pegon} for smooth functions.

If $f\in W^{1,p}(\Sd)$, by the density of $\mathcal{C}^1(\Sd)$ in $W^{1,p}(\Sd)$ one can choose a sequence $f_j$ of functions in $\mathcal{C}^1(\Sd)$ which converges to $f$ in $W^{1,p}(\Sd)$, and pointwise. Then, writing \eqref{pointwise bound pegon} for $f_j$, taking the limit in $j$ and applying Fatou's lemma gives the result for $W^{1,p}(\Sd)$.

Finally, if $ f \in \mbox{BV}(\mathbb S^d)$, there exists a sequence $\{f_j\}$ with the properties prescribed in \eqref{miranda approximation}. One can further assume that $f_j$ converges to $f$ pointwise. Hence,  writing \eqref{pointwise bound pegon} for $f_j$, taking the limit in $j$ and applying Fatou's lemma gives the result for $\mbox{BV}(\mathbb S^d)$.
\end{proof}
We are now in the position to prove the norm representation result on the unit sphere.

\begin{proof}[Proof of Theorem~\ref{thm:davila sphere}]
Assume first that $f\in\mathcal{C}^1(\Sd)$. We fix $\delta\in (0,\pi)$, and write $I_{d,p}(\varrho_L, f)$ as a sum of a local and a global term,
\[
    I_{d,p}(\varrho_L, f)=I_{d,p}(\varrho_L\chi_{(0,\delta)}, f)+I_{d,p}(\varrho_L\chi_{(\delta,\infty)}, f).
\]
For the global term we have
\[
    I_{d,p}(\varrho_L\chi_{(\delta,\infty)}, f)\leq \frac{2\Vert f\Vert_\infty^p}{\delta^p}\int_{\Sd}\int_{\Sd\setminus B(y,\delta)}\varrho_L(d(x,y))\,  d\sigma(x)d\sigma(y),
\]
which, by the properties of the mollifiers, tends to zero as $L$ tends to infinity. Then, it is enough to prove that the local term tends to $K_{d,p}\int_{\Sd}|\nabla f|^p$ as $L\to\infty$.

By \eqref{pegon formula} we can write
\begin{equation}\label{eq: local integral sphere}
    I_{d,p}(\varrho_L\chi_{(0,\delta)}, f)=\iint_{d(x,y)<\delta}\Big(|\nabla f(y)\cdot e_{xy}(0)|^p +F_p(x,y)\Big)\varrho_L(d(x,y))\,  d\sigma(x)d\sigma(y),
\end{equation}
where
\[
        F_p(x,y)=\Big|\int_0^1 \nabla f(\gamma_{yx}(t))\cdot e_{xy}(t)\, dt\Big|^p-|\nabla f(y)\cdot e_{xy}(0)|^p=:|b|^p-|a|^p.
\]
Using the same notation as in the proof of Proposition \ref{modified pegon} for spherical changes of variable, if $x=\Phi_y(\xi,d(x,y))$, then we have $e_{xy}(0)=\Phi_y(\xi,\pi/2)$, so that
\[
\begin{split}
        &\iint_{d(x,y)<\delta}|\nabla f(y)\cdot e_{xy}(0)|^p \varrho_L(d(x,y))\,  d\sigma(x)d\sigma(y)\\
        &=\frac{\omega_{d-1}}{\omega_d}\int_{\Sd}\int_0^\delta\int_{\mathbb S^{d-1}}\Big|\nabla f(y)\cdot \Phi_y(\xi,\pi/2)\Big|^pd\sigma(\xi)\varrho_L(r)\sin^{d-1}r \, dr d\sigma(y)\\
        &=K_{d,p}\int_{\Sd}|\nabla f(y)|^p\, d\sigma(y)\Big(\frac{\omega_{d-1}}{\omega_d}\int_0^\delta\varrho_L(r)\sin^{d-1}r\, dr \Big)\\
        &=K_{d,p}\int_{\Sd}|\nabla f(y)|^p\, d\sigma(y)\int_{B(N,\delta)}\varrho_L(d(x,N))d\sigma(x)\longrightarrow K_{d,p}\int_{\Sd}|\nabla f(y)|^p\, d\sigma(y), \quad L\to\infty,
\end{split}
\]
where $N$ is an arbitrarily chosen point on $\Sd$ and the limit is justified by \eqref{def_mollifier}.
In view of \eqref{eq: local integral sphere}, to prove the result for smooth functions it is then enough to show that
\begin{equation}\label{eq:F_p limit}    \lim_{L\to\infty}\lim_{\delta\to\infty}\iint_{d(x,y)<\delta}F_p(x,y)\varrho_L(d(x,y))\,  d\sigma(x)d\sigma(y)=0.
\end{equation}
We will actually prove that the above holds for any $L$, without even taking the limit in $L$. Recalling that $e_{xy}(t)$ is the normalized tangent vector of the geodesic $\gamma_{yx}$ at time $t$, by Cauchy–Schwarz inequality both $|a|$ and $|b|$ are trivially bounded by $\|\nabla f\|_{\infty}$. Then, by the fundamental theorem of calculus and Jensen's inequality, we get
\[
\begin{split}
      |F_p(x,y)|&=\Big|\int_a^b\Big(\frac{d}{d\xi}|\xi|^p\Big)d\xi\Big|=\Big|p\int_a^b\xi|\xi|^{p-2}d\xi\Big|\\
      &\leq p\int_a^b|\xi|^{p-1}\, d\xi
    =p|b-a|\int_0^1 |(1-\lambda)a+\lambda b|^{p-1}\, d\lambda\\
    &\leq p|b-a|\int_0^1 (1-\lambda)|a|^{p-1}+\lambda |b|^{p-1}d\lambda\leq p|b-a|\|\nabla f\|_{\infty}^{p-1}.
\end{split}
\]
The term $|b-a|$ can be estimated, by triangular and Cauchy–Schwarz inequalities, by
\[
\begin{split}
      |b-a|&=  \Big|\int_0^1 \big(\nabla f(\gamma_{yx}(t))\cdot e_{xy}(t)-\nabla f(y)\cdot e_{xy}(0)\big)\, dt\Big| \\
      &\leq \int_0^1 |\nabla f(\gamma_{yx}(t))-\nabla f(y)|+ |\nabla f(y)||e_{xy}(t)-e_{xy}(0)| \, dt.
\end{split}
\]
By the uniform continuity of $\nabla f$, for any $\varepsilon>0$, we can choose $\delta$ small enough so that, for $d(x,y)<\delta$,
\[
    \max\{|e_{xy}(t)-e_{xy}(0)|\|\nabla f\|_\infty,  |\nabla f(\gamma_{yx}(t))-\nabla f(y)|\}<\frac{\varepsilon}{2},
\]
which gives $|b-a|<\varepsilon$ and, in turn, $|F_p(x,y)|<p\varepsilon\|\nabla f\|_{\infty}^{p-1}$. Hence, recalling that for every $L$ and every $y\in \Sd$, $\int_{\Sd}\varrho_L(d(x,y))\,  d\sigma(x)=1$, we get
\[
    \iint_{d(x,y)<\delta}|F_p(x,y)| \varrho_L(d(x,y))\,  d\sigma(x)d\sigma(y)<p\varepsilon\|\nabla f\|_{\infty}^{p-1}.
\]
By the arbitrariness of $\varepsilon$, we get \eqref{eq:F_p limit}, and hence the desired result for smooth functions, namely,
\begin{equation}\label{pegon for smooth functions}
    \lim_{L\to \infty}I_{d,p}(\varrho_L, f)=K_{d,p}\int_{\Sd} |\nabla f(z)|^pd\sigma(z), \quad f\in\mathcal{C}^1(\Sd).
\end{equation}
Let now $f\in L^p(\Sd)$, $1\leq p<\infty$, and define the regularized functions
\[
f_\varepsilon(x)=\int_{SO(d+1)} \phi_\varepsilon(\nu N)f(\nu^{-1}x)\, dH(\nu),
\]
    where $\phi_\varepsilon$ are smooth functions supported on a spherical cap of radius $\varepsilon>0$ centered at the North Pole $N\in \Sd$, depending only on the distance to the North Pole. They are normalized to have $\int_{SO(d+1)} \phi_\varepsilon(\nu N)dH(\nu)=1$, where $H$ is the Haar measure on $SO(d+1)$, and we can assume that $f_\varepsilon\longrightarrow f$ in $L^p(\Sd)$ as $\varepsilon\to 0$. Then we get,

\begin{align*}
I_{d,p}(\varrho_L, f_\varepsilon)&
\le
\iint_{\Sd\times\Sd}\frac{\varrho_L(d(x,y))}{d(x,y)^p}\int_{SO(d+1)} 
\phi_\varepsilon(\nu N)|f(\nu^{-1}x)-f(\nu^{-1}y)|^p \, dH(\nu)  d\sigma(x)d\sigma(y)    
\\
&
=
\int_{SO(d+1)} \phi_\varepsilon(\nu N) \iint_{\Sd\times\Sd}\frac{|f(\nu^{-1}x)-f(\nu^{-1}y)|^p}{d(x,y)^p} 
 \varrho_L(d(x,y)) \, d\sigma(x)d\sigma(y)     dH(\nu)
\\
&
=
\iint_{\Sd\times\Sd}\frac{|f(x)-f(y)|^p}{d(x,y)^p} \varrho_L(d(x,y))
  \, d\sigma(x)d\sigma(y)=I_{d,p}(\varrho_L, f),   
\end{align*}
because $d(\nu^{-1}x,\nu^{-1}y)=d(x,y)$ 
and the measure $\sigma$ is invariant under rotations. Taking the limit in $L$ on both sides, and using \eqref{pegon for smooth functions} for the smooth functions $f_\varepsilon$, we get
\begin{equation}\label{eq:smooth bound}
    K_{d,p}\int_{\Sd}|\nabla f_\varepsilon|^p\leq \liminf_{L\to\infty}I_{d,p}(\varrho_L, f).
\end{equation}
Observe that if the right hand-side of \eqref{eq:smooth bound} is finite, then the sequence $f_\varepsilon$ is uniformly bounded in $W^{1,p}(\Sd)$. Then, if $1<p<\infty$, there exists a subsequence of $f_\varepsilon$ admitting a weak limit, which must coincide with $f$. This proves that the limit in \eqref{limit_pegon_p} is not finite if $1<p<\infty$ and $f\not \in W^{1,p}(\Sd)$. Suppose then that $f\in W^{1,p}(\Sd)$, $1\leq p<\infty$. In this case $(\nabla f)_\varepsilon$ converges to $\nabla f$ in $L^p(\Sd)$ and, passing to a subsequence, pointwise. Then, by Fatou's lemma and \eqref{eq:smooth bound},
\[
    K_{d,p}\int_{\Sd} |\nabla f|^p\leq  K_{d,p}\liminf_{\varepsilon\to 0}\int_{\Sd} |\nabla f_\varepsilon|^p\leq\liminf_{L\to\infty}I_{d,p}(\varrho_L, f).
\]
Proposition~\ref{modified pegon} applied to each mollifier $\varrho_L$ gives the reverse bound
    \[
        \limsup_{L\to \infty}I_{d,p}(\varrho_L, f)\le K_{d,p}\int_{\Sd} |\nabla f|^p,
    \]
which completes the proof of \eqref{limit_pegon_p}.

Let now $p=1$, and $\phi$ any $C^\infty_c$ one-form with $\|\phi\|\le 1$. Since $f_\varepsilon\longrightarrow f$ in $L^1(\Sd)$, then
    \[
        \int_{\Sd} f(x) \Div \phi(x)\,  dV(x)=\lim_{\varepsilon\to 0}\int_{\Sd} f_\varepsilon(x) \Div \phi(x)\,  dV(x)\leq \liminf_{\varepsilon\to 0}[f_\varepsilon]_{BV}.
    \]
Taking the supremum over all $\phi$ as above we get
\[
    K_d[f]_{BV}\leq \liminf_{L\to\infty}I_{d,p}(\varrho_L, f),
\]
which shows that the limit in \eqref{limit_pegon_BV} is not finite if $f\not \in BV(\Sd)$. Moreover, if $f\in BV(\Sd)$, applying Proposition~\ref{modified pegon} to each mollifier $\varrho_L$, we obtain the reverse bound
    \[
        \limsup_{L\to \infty}I_{d,1}(\varrho_L, f)\le K_{d}[f]_{BV(\Sd)},
    \]
which completes the proof of \eqref{limit_pegon_BV}.
\end{proof}

\subsection{Harmonic ensemble: rough statistics}

Given Theorem~\ref{thm:davila sphere}, in order to prove Theorem~\ref{variance asymptotic rough} it is enough to show that if the variance of the rough statistics of the harmonic ensemble is written as in \eqref{eq:general rough} with $C_{rough}(N)=C_d N^{1-1/d}\log N$, then the functions $\varrho_N$ appearing in the right-hand side of \eqref{eq:general rough} are radial mollifiers on $\Sd$.

\begin{proof}[Proof of Theorem~\ref{variance asymptotic rough}]

Let $K_L(x,y)=K_L(d(x,y))$ be the kernel of the harmonic ensemble on $\Sd$ with $N$ points, where $N=\pi_L$ and $\pi_L$ is defined as in \eqref{eq:piL}. Then,
\begin{equation}\label{N to L harmonic rough}
    \lim_{N\to\infty} \frac{\var( n_A )}{N^{1-\frac{1}{d}}\log N}=\frac{1}{d}\left( \frac{\Gamma(d+1)}{2} \right)^{1-\frac{1}{d}}\lim_{L\to \infty}\frac{\var( n_A )}{L^{d-1}\log L}.
\end{equation}
By \eqref{equation_variance} we have 
for $A\subset \mathbb S^d$ a Caccioppoli set
\begin{align*}
\var( n_A )&= \frac{1}{2}\iint_{\Sd\times\Sd}|\chi_A(x)-\chi_A(y)| K_L(d(x,y))^2 \, d\sigma(x)d\sigma(y)
\\       
&
=
\frac{C_L}{2}
\iint_{\Sd\times\Sd}\frac{|\chi_A(x)-\chi_A(y)|}{d(x,y)} \varrho_L(x,y) \, d\sigma(x)d\sigma(y),
\end{align*}
where $\varrho_L(x,y)=C_{L}^{-1} K_L(d(x,y))^2  d(x,y)$ and
\[
C_L=\int_{\mathbb S^d}K_L(d(x,y))^2  d(x,y) d\sigma(x).
\]

It is enough to show that $\{\varrho_L\}_{L\ge 0}$ is a sequence of radial mollifiers on $\Sd$ as in \eqref{def_mollifier} and that
    \begin{equation}\label{C_L asymptotic harmonic}
        \lim_{L\to \infty} \frac{C_L}{L^{d-1}\log L}=\frac{4\omega_{d-1}\Gamma(\frac{d}{2}+1)^2}{ \pi \omega_{d} \Gamma(d+1)^2}=:B_d,
    \end{equation}
because then Theorem~\ref{thm:davila sphere}, together with \eqref{N to L harmonic rough}, implies the desired result,
\[
\lim_{N\to\infty}\frac{\var( n_A )}{N^{1-\frac{1}{d}}\log N}=
\frac{1}{d}\left( \frac{\Gamma(d+1)}{2} \right)^{1-\frac{1}{d}}\frac{ B_d K_d}{2} [\chi_A]_{BV}.
\]
We first prove \eqref{C_L asymptotic harmonic}. By rotation invariance and Szeg\"o's estimate \eqref{szego interval} we have
\[
C_L=\frac{\pi_L^2 }{ \binom{L+d/2}{L}^2}\frac{\omega_{d-1}}{\omega_d}\int_0^\pi \theta P_L^{(1+\lambda,\lambda)}(\cos \theta)^2 \sin^{d-1} \theta\,  d\theta.
\]
We split the integral on the right in the sum of the three integrals on the intervals $[0,1/L]$, $[1/L,\pi-1/L]$, and $[\pi-1/L,\pi]$. The first and the third are of the order $\bigO{L^{-1}}$, as one can easily see  using the well known bounds $P_L^{(1+\lambda,\lambda)}(t)^2\lesssim L^{d-2}$ for $-1\le t\le 0$ and
$P_L^{(1+\lambda,\lambda)}(t)^2\lesssim L^{d}$ for $0\le t\le 1$, see \cite[7.32]{szego}.

For the middle term, by \eqref{szego interval} and \eqref{h_L} we have
\begin{align*}
&\int_{\frac{1}{L}}^{\pi-\frac{1}{L}} \theta P_L^{(1+\lambda,\lambda)}(\cos \theta)^2 \sin^{d-1} \theta\,  d\theta=\frac{2}{\pi L}\int_{\frac{1}{L}}^{\pi-\frac{1}{L}} 
\frac{\theta}{\sin^2 \frac{\theta}{2}}h_L^2(\theta)\,  d\theta\\
&=
\frac{1}{\pi L}\int_{\frac{1}{L}}^{\pi-\frac{1}{L}}  
\frac{\theta}{\sin^2 \frac{\theta}{2}}\,  d\theta
+
\frac{1}{\pi L}\int_{\frac{1}{L}}^{\pi-\frac{1}{L}}  
\frac{\theta}{\sin^2 \frac{\theta}{2}} \cos\bigl(2(L+\frac{d}{2})\theta-c_d\bigr)\,  d\theta
\\
&
+\frac{2}{\pi L}
\int_{\frac{c}{L}}^{\pi-\frac{c}{L}}  
\frac{\theta}{\sin^2 \frac{\theta}{2}} \frac{\cos\bigl(2(L+\frac{d}{2})\theta-c_d\bigr)}{L \sin \theta}\,  d\theta+o(1),
\end{align*}
as $L\to\infty$. For the first summand on the right, using
\begin{equation}\label{integral_middle}
\int_{\varepsilon}^{\pi-\varepsilon}\frac{\theta}{\sin^2 \frac{\theta}{2}} \, d\theta=
4+\log 2 -\pi \varepsilon-\log \varepsilon+\bigO{\varepsilon^2}, \quad \varepsilon\to 0,
\end{equation}
we obtain that 
\[
\frac{1}{\pi L}\int_{\frac{1}{L}}^{\pi-\frac{1}{L}}  
\frac{\theta}{\sin^2 \frac{\theta}{2}}\, d\theta=\frac{1}{\pi}\frac{\log L}{L}+\bigO{1}, 
\quad L\to +\infty.
\]
The second integral can be shown to be $\bigO{L}$ when $L\to +\infty$
integrating by parts and the last one is again of the same order. All together,
\[
C_L\sim \frac{1}{\pi}\frac{\pi_L^2 }{ \binom{L+d/2}{L}^2}\frac{\omega_{d-1}}{\omega_d}\frac{\log L}{L}\sim \frac{4}{\pi}\frac{\Gamma(d/2+1)^2}{\Gamma(d+1)^2}\frac{\omega_{d-1}}{\omega_d}L^{d-1}\log L,
\]
as desired.

It remains to prove that $\{\varrho_L\}_{L\ge 0}$ is a sequence of radial mollifiers. Condition $(i)$ in \eqref{def_mollifier} is immediate. To prove $(ii)$ we let $\delta>0$ be fixed. Then,
\[
\int_{B(y,\delta)^c}\varrho_{L}(x,y)\, d\sigma(x)=
\frac{1}{\log L}\frac{\omega_{d-1}}{\omega_d}\int_{\delta}^{\pi}  \theta P_L^{(1+\lambda,\lambda)}(\cos \theta)^2 \sin^{d-1} \theta\, d\theta,
\]
which, using exactly the same estimates as above, we to be asymptotically of the order $L^{-1}$, and therefore tends to zero as $L\to\infty$. This proves (ii) and completes the proof.
\end{proof}

With a similar computation as above, one can see that it is not possible to apply Theorem~\ref{thm:davila sphere} to study the smooth statistics. Indeed, from Theorem~\ref{variance asymptotic} we know that the order of the smooth statistic is $L^{d+1}$, but the functions $\varrho_L(x,y):=L^{-(d-1)}K_L(x,y)^2d(x,y)^2$ are not mollifiers because 
\[
\lim_{L \to +\infty}\int_{B(x,\delta)^c}\varrho_L(x,y)\, d\sigma(y)\approx 1.
\]

\subsection{Spherical ensemble: smooth and rough statistics}

To determine the asymptotic behavior of the linear statistics for this point process is now simple. It is enough to prove the following result, which shows that we are in the position of applying Theorem~\ref{thm:davila sphere} both in the smooth and  in the rough case.

\begin{lem}\label{lem: mollifiers spherical}
Let $K_N(x,y)=K_N(d(x,y))$ be the kernel of the $N$ points spherical ensemble on $\mathbb S^2$. Let, for $p\in \{1,2\}$,
\[
C_{p,N}=\Big(\frac{1}{\sqrt{N}}\Big)^{2-p}\int_{\mathbb S^2}K_N(x,y)^2  d(x,y)^{p}\, d\sigma(x).
\]
Then,
\[
    \lim_{N\to\infty} C_{1,N}=\sqrt{\pi}, \quad \lim_{N\to\infty} C_{2,N}=4,
\]
and,
    \[
       \varrho_{p,N}(r)= \frac{1}{C_{p,N}}\Big(\frac{1}{\sqrt{N}}\Big)^{2-p}K_N(r)^2  r^p, \quad N>0,
    \]
is a sequence of radial mollifiers on $\mathbb S^2$.
\end{lem}

\begin{proof}
Clearly, for every $N>0$, the function $\varrho_N$ satisfies $(i)$ in the definition of radial mollifiers \eqref{def_mollifier}. Observe that for $g(x)=z,g(y)=w \in \mathbb C$, where $g$ stands for the stereographic projection,
\[
d(x,y)=2\arctan \left| \frac{z-w}{1+z\bar{w}}\right|.
\]
Now for any $y\in \mathbb S^2$ (without loss of generality we assume that $g(y)=0$),
\begin{equation}\label{global part spherical}
\begin{split}
    \int_{d(x,y)>\delta}K_N(x,y)^2  d(x,y)^{p} d\sigma(x)&=2^{p+1}N^2 \int_{|z|>\tan\frac{\delta}{2}} \frac{\arctan^p|z| }{(1+|z|^2)^{N-1}}\frac{dz}{\pi (1+|z|^2)^2}\\
    &=2^{p+1}N^2\int_{\tan\frac{\delta}{2}}^{+\infty} \frac{r \arctan^p{r}}{(1+r^2)^{N+1}}dr.
\end{split}  
\end{equation}
It follows that
\[
    \Big(\frac{1}{\sqrt{N}}\Big)^{2-p}\int_{d(x,y)>\delta}K_N(x,y)^2  d(x,y)^{p} d\sigma(x)\lesssim N^{1+\frac{p}{2}}\int_{\tan\frac{\delta}{2}}^{+\infty} \frac{r \ dr}{(1+r^2)^{N+1}}=\frac{N^{p/2}}{(\tan^2(\delta/2)+1)^N},
\]
which tends to $0$ as $N\to\infty$. Then, to prove $(ii)$, since
\[
    \int_{d(x,y)>\delta}\varrho_{p,N}(d(x,y)) d\sigma(x)=\frac{1}{C_{p,N}}\Big(\frac{1}{\sqrt{N}}\Big)^{2-p} \int_{d(x,y)>\delta}K_N(x,y)^2  d(x,y)^{p} d\sigma(x),
\]
it is enough to prove that $C_{p,N}$ converges to a constant as $N\to\infty$. By taking the limit as $\delta\to 0$ in \eqref{global part spherical}, we have
\[
C_{p,N}=2^{p+1} N^{1+p/2} \int_{0}^{+\infty} \frac{r \arctan^p{r} }{(1+r^2)^{N+1}}\,
dr.
\]
Integrating by parts,
\[
\int_{0}^{+\infty} \frac{r \arctan{r} }{(1+r^2)^{N+1}}\, 
dr=\frac{\sqrt{\pi} \Gamma(N+\frac{1}{2})}{4N\Gamma(N+1)}\sim \frac{\sqrt{\pi}}{4N^{3/2}}, \quad N\to\infty,
\]
and integrating by parts again,
\[
\int_{0}^{+\infty} \frac{r \arctan^2{r} }{(1+r^2)^{N+1}}
dr=\frac{1}{N}\int_{0}^{+\infty} \frac{ \arctan {r} }{(1+r^2)^{N+1}}\, 
dr\sim \frac{1}{2N^2}.
\]
To justify the last estimate, observe that the integral
$\int_{\varepsilon}^{+\infty} \frac{ \arctan {r}}{(1+r^2)^{N+1}}\, dr$ is exponentially small
for $\varepsilon>0$ and the integral 
$\int_{0}^{\varepsilon} \frac{ \arctan {r}}{(1+r^2)^{N+1}}\, dr$
can be estimated with $\arctan {r}\sim r$. Then, it is proved that
\[
    \lim_{N\to\infty} C_{1,N}=\sqrt{\pi}, \quad \lim_{N\to\infty} C_{2,N}=4,
\]
which also implies that $\varrho_{p,N}$ is a family of radial mollifiers and concludes the proof.
\end{proof}

\begin{proof}[Proof of Theorem~\ref{variance asymptotic spherical smooth}]
    Applying \eqref{equation_variance} to $f \in H^1(\mathbb S^2)$, we obtain
\[
\var\Big( \sum_{i=1}^N f(x_i)\Big)=\frac{C_{2,N}}{2}\iint_{\mathbb S^2\times \mathbb S^2}
\frac{|f(x)-f(y)|^2}{d(x,y)^2}\varrho_{2,N}(d(x,y))\, d\sigma(x)d\sigma(y),  
\]
where $C_{2,N}$ and $\varrho_{2,N}$ are defined as in Lemma~\ref{lem: mollifiers spherical}. Then, applying Lemma~\ref{lem: mollifiers spherical} for $p=2$ and Theorem~\ref{thm:davila sphere}, and recalling that $K_{2,2}=1/2$, we get the result.
\end{proof}

\begin{proof}[Proof of Theorem~\ref{variance asymptotic spherical rough}]
    Applying \eqref{equation_variance} to the characteristic function of $A\subset \mathbb S^2$ we get
\[
\frac{\var(n_A)}{\sqrt{N}}=\frac{C_{1,N}}{2}\iint_{\mathbb S^2\times \mathbb S^2}
\frac{|\chi_A(x)-\chi_A(y)|}{d(x,y)} \varrho_{1,N}(d(x,y))\, d\sigma(x)d\sigma(y).
\]
where $C_{1,N}$ and $\varrho_{1,N}$ are defined as in Lemma~\ref{lem: mollifiers spherical}. Then, applying Lemma~\ref{lem: mollifiers spherical} for $p=1$ and Theorem~\ref{thm:davila sphere}, and recalling that $K_{2,1}=2/\pi$, we get the result.
\end{proof}

\subsection{Norm representations in \texorpdfstring{$\mathbb R^d$}{Euclidean space}}\label{sec: norm representation euclidean space}
Given a domain $\Omega\subset \R^d$, $1\leq p<\infty$, and a family $\varrho_n: \R^d\times \R^d\to \R^+$, we say that $\varrho_n$ are asymptotically radial $p$-mollifiers in $\Omega$ if they satisfy
\begin{enumerate}
    \item[(a)] For any $\delta>0$ and any $x\in \Omega$, $\lim_{n \to\infty} \int_{B(x,\delta)^c} \frac{\varrho_n(x,y)}{|x-y|^p} dy = 0$, uniformly on compact sets $K\subset \Omega$.
    \item [(b)]For any $x\in \Omega$, $\lim_{n\to\infty}\int_{B(x,1)} \varrho_n(x,y)\, dy = 1$ uniformly on compact sets  $K\subset \Omega$.
    \item[(c)] For any $x\in \Omega$  there are functions $h_n:\R^+\to \R$, such that
     \[\lim_{n\to\infty}\int_{B(x,1)} \frac{|\varrho_n(x,y)-h_n(|x-y|)|}{|x-y|^p}\, dy = 0,\]
     uniformly on compact sets $K\subset \Omega$.
    \item[(d)] They are symmetric $\varrho_n(x,y) = \varrho_n(y,x)$.
\end{enumerate}
Condition (a) indicates that $\varrho_n(x,y)$ are concentrated in the diagonal, while condition (c) indicates that near the diagonal the mollifiers $\varrho_n$ are close to radial. In the most classical case, i.e., when $\varrho_n$ are functions of $|x-y|$, the hypothesis (c) and (d) are empty, and the above conditions reduce to (a) and (b). In this restricted setting, Theorem~\ref{thm:asymptotically radial mollifiers} is well known: see \cite{brezis} for the case in which the $\varrho_n$ are radial mollifiers, i.e., functions satisfying \eqref{def_mollifier}, and \cite{Gounoue} for the case in which, more in general, the $\varrho_n$ are radial functions satisfying (a) and (b).

\begin{proof}[Proof of Theorem~\ref{thm:asymptotically radial mollifiers}]
Let $f\in L^p(\R^d)$ be a function with compact support $K$, and let $\Omega$ be a $\delta$-neighborhood of $K$, for some small $\delta>0$. Clearly,
\[\int_{\Omega^c}\int_{B(x,\delta)}\frac{|f(x)-f(y)|^p}{|x-y|^p}\varrho_n(x,y) dydx=0.
\]
Moreover, by discrete Jensen's inequality and (d) it is easily seen that
\[\int_{\R^d}\int_{B(x,\delta)^c}\frac{|f(x)-f(y)|^p}{|x-y|^p}\varrho_n(x,y) dydx\le 2^p\int_{\R^d}|f(x)|^p\int_{B(x,\delta)^c} \frac{\varrho_n(x,y)}{|x-y|^p}\, dydx,
\]
which by property (a), the fact that $f\in L^p(\R^d)$ with compact support, and the dominated convergence theorem, tends to zero as $n\to\infty$. Then we have

\[
\iint_{\R^d\times \R^d}\frac{|f(x)-f(y)|^p}{|x-y|^p}\varrho_n(x,y)\, dxdy=\int_{\Omega}\int_{B(x,\delta)}\frac{|f(x)-f(y)|^p}{|x-y|^p}\varrho_n(x,y)\, dydx+o(1).
\]
Reasoning as above,
\[
\begin{split}
    &\int_{\Omega}\int_{B(x,\delta)^c}\frac{|f(x)-f(y)|^p}{|x-y|^p}|\varrho_n(x,y)-h_n(|x-y|)| dydx\\
    &\le 2^p\int_{K}|f(x)|^p\int_{B(x,\delta)^c} \frac{|\varrho_n(x,y)-h_n(|x-y|)|}{|x-y|^p}\, dydx,
\end{split}
\]
which by property (c), the fact that $f\in L^p(\R^d)$, and the dominated convergence theorem, tends to zero as $n\to\infty$, uniformly in $x\in K$. Then,
\begin{equation}\label{rho to h double integral}
\iint_{\R^d\times \R^d}\frac{|f(x)-f(y)|^p}{|x-y|^p}\varrho_n(x,y)\, dxdy=\int_{\Omega}\int_{B(x,1)}\frac{|f(x)-f(y)|^p}{|x-y|^p}h_n(|x-y|) dydx+o(1).
\end{equation}

Suppose now that $f\in \mathcal{C}^2_c(\R^d)$ with support in $K$. Reasoning as in  \cite[Theorem 2]{BBM}, we get that for any $x\in\Omega$,
\begin{equation}
\begin{split}\label{eq:bbm}
&\int_{B(x,1)} \frac{|f(x)-f(y)|^p}{|x-y|^p}h_n(|x-y|)\, dy\\
&=K_{d,p} |\nabla f(x)|^p\int_{B(x,1)} h_n(|x-y|)\, dy + \mathcal{O}\left(\int_{B(x,1)} |x-y|^p h_n(|x-y|)\, dy\right).
\end{split}
\end{equation}
Observe that for any $\varepsilon\in (0,1)$, by property (c)
\[
\begin{split}
    \int_{B(x,1)}& |x-y|^ph_n(|x-y|)\,  dy \\
    & =\int_{B(x,\varepsilon)} |x-y|^ph_n(|x-y|)\,  dy+ \int_{\varepsilon\le|x-y|<1} |x-y|^ph_n(|x-y|)\,  dy\\
    &\leq \varepsilon^p \int_{B(x,\varepsilon)} h_n(|x-y|) \, dy+\int_{\varepsilon\le|x-y|<1} \frac{h_n(|x-y|)}{|x-y|^p}\,  dx\\
    &= \varepsilon^p \int_{B(x,\varepsilon)} \varrho_n(x,y) \, dy+\int_{\varepsilon\le|x-y|<1} \frac{\varrho_n(x,y)}{|x-y|^p}\,  dx+o(1),
\end{split}
\]
which by properties (a), (b), and the arbitrariness of $\varepsilon$, tends to zero uniformly on compact sets. Thus, the second summand in \eqref{eq:bbm} tends to 0, while the first summand tends to $K_{d,p}|\nabla f(x)|^p$ thanks to property (b), again uniformly on compact sets. Therefore, we have seen that 
\begin{equation}
\begin{split}
    &\lim_{n\to\infty} \int_{B(x,1)} \frac{|f(x)-f(y)|^p}{|x-y|^p}\varrho_n(x,y) \, dy\\
    &=\lim_{n\to\infty}\int_{B(x,1)} \frac{|f(x)-f(y)|^p}{|x-y|^p}h_n(|x-y|)\, dy =K_{d,p}|\nabla f(x)|^p.
\end{split}   
\end{equation}
Moreover, the above integral is uniformly bounded on compact sets by (b), because $|f(x)-f(y)|\le L|x-y|$ as $f$ is Lipschitz. Then, by the dominated convergence theorem, we have
\begin{equation}\label{qmol}
\lim_{n\to\infty} \iint_{\R^d\times \R^d}\frac{|f(x)-f(y)|^p}{|x-y|^p}\varrho_n(x,y)\, dxdy =K_{d,p} \int_{\R^d} |\nabla f|^p, \quad f\in \mathcal C^2_c(\mathbb R^d),
\end{equation}
which proves the result for smooth functions.

Let us now discuss how to extend the result to non-smooth functions. Suppose that $f$ is a general compactly supported function in $L^\infty(\R^d)$, and define the regularized functions
\[
f_\varepsilon(x)=\int_{\R^d} \phi_\varepsilon(z)f(x-z) dz,
\]
where $\phi_\varepsilon$ are smooth, positive, and compactly supported approximation of the unity functions that depending only on the distance to the origin, and normalized to have $\int_{\R^d} \phi_\varepsilon(z)dz = 1$. Since $f$ has compact support, $f_\varepsilon$ is smooth and has compact support, and we can assume that $K$ contains the support of $f$ as well as that of all the regularized functions $f_\varepsilon$. Then, by \eqref{qmol}, \eqref{rho to h double integral} and Jensen's inequality we have

\[
        \begin{split}
    &K_{d,p}\int_{\R^d}|\nabla f_\varepsilon|^p=\lim_{n\to\infty}\int_{\Omega}\int_{B(x,1)}\frac{|f_\varepsilon(x)-f_\varepsilon(y)|^p}{|x-y|^p}h_n(|x-y|) dydx\\
    &\le\liminf_{n\to\infty}\int_{\Omega}\int_{\R^d} 
\phi_\varepsilon(z)\int_{B(x,1)}\frac{|f(x-z)-f(y-z)|^p}{|x-y|^p}h_n(x-y)\, dydzdx\\
&=\liminf_{n\to\infty}\int_{\Omega}\int_{B(x,1)}\frac{|f(x)-f(y)|^p}{|x-y|^p}h_n(x-y)\, dydx\\
&=\liminf_{n\to\infty}\iint_{\R^d\times \R^d}\frac{|f(x)-f(y)|^p}{|x-y|^p}\varrho_n(x,y)\,  dydx.
\end{split}
\]
Reasoning exactly as in the proof of Theorem~\ref{thm:davila sphere}, we have that the inferior limit in the last line being finite implies $f\in W^{1,p}(\R^d)$, if $1<p<\infty$, and $f\in BV(\R^d)$, if $p=1$. Moreover, if $f\in W^{1,p}(\R^d)$ for some $1\leq p<\infty$,
\[
    K_{d,p}\int_{\R^d}|\nabla f|^p\leq \liminf_{n\to\infty}\iint_{\R^d\times\R^d}\frac{|f(x)-f(y)|^p}{|x-y|^p}\varrho_n(x,y)\,  dydx,
\]
and if $f\in BV(\R^d)$
\[
    K_d[f]_{BV}\leq \liminf_{n\to\infty}\iint_{\R^d\times\R^d}\frac{|f(x)-f(y)|}{|x-y|}\varrho_n(x,y)\,  dydx.
\]
It remains to show the reverse inequality for the superior limit. Choose a sequence $\{f_j\}$ of smooth and compactly supported functions which converges to $f$ pointwise and such that
\[
\lim_j\int_{\R^d}|\nabla f_j|^p=
    \int_{\R^d}|\nabla f|^p, \quad \text{or} \quad\lim_j\int_{\R^d}|\nabla f_j| =[f]_{BV(\R^d)},
\]
if, respectively, $f\in W^{1,p}(\R^d)$ or $f\in BV(\R^d)$. Recall that such a sequence exists by the density of smooth compactly supported functions in $W^{1,p}(\R^d)$, or by \eqref{miranda approximation}, respectively. By \eqref{eq:bbm}  we have that
\begin{equation*}
\int_{B(x,1)} \frac{|f_j(x)-f_j(y)|^p}{|x-y|^p}h_n(|x-y|)\, dy = K_{d,p} |\nabla f_j(x)|^p\int_{B(x,1)} h_n(|x-y|)\, dy +o(1).
\end{equation*}
We denote 

\[
C_n=\int_{B(x,1)} h_n(|x-y|)\, dy.
\]
Then, by \eqref{rho to h double integral}, Fatou's lemma and \eqref{qmol} we get
\[
\begin{split}
       &\iint_{\R^d\times\R^d}\frac{|f(x)-f(y)|^p}{|x-y|^p}\varrho_n(x,y)\,  dydx= \int_{\Omega}\int_{B(x,1)} \frac{|f(x)-f(y)|^p}{|x-y|^p}\varrho_n(x,y)\, dydx +o(1)\\
       &\le\liminf_{j\to\infty}\int_{\Omega}\int_{B(x,1)} \frac{|f_j(x)-f_j(y)|^p}{|x-y|^p}\varrho_n(x,y)\, dy + o(1) \\
       &=\liminf_{j\to\infty}\int_{\Omega}\int_{B(x,1)} \frac{|f_j(x)-f_j(y)|^p}{|x-y|^p}h_n(|x-y|)\, dy + o(1) \\
       &=  \lim_{j\to\infty} C_n K_{d,p} \int_{\R^d}|\nabla f_j|^p +o(1).
\end{split}
\]
As $\lim_n C_n = 1$, it follows that
\[
\limsup_{n\to\infty}\iint_{\R^d\times\R^d}\frac{|f(x)-f(y)|^p}{|x-y|^p}\varrho_n(x,y)\,  dydx \leq K_{d,p} \int_{\R^d}|\nabla f|^p,
\]
when $f\in W^{1,p}(\R^d)$, $1\leq p<\infty$, and
\[
\limsup_{n\to\infty}\iint_{\R^d\times\R^d}\frac{|f(x)-f(y)|^p}{|x-y|^p}\varrho_n(x,y)\,  dydx \leq K_{d,p} [f]_{BV(\R^d)},
\]
when $f\in BV(\R^d)$. This completes the proof.
\end{proof}

\subsection{Ginibre process: smooth and rough statistics}
The infinite Ginibre process has a radial kernel. In this case, we can deduce the asymptotic behavior of its (smooth and rough) linear statistics by means of the norm representation results by Brezis \cite{brezis}. In order to apply these results, we need to show that the squared modulus of the kernel of the process, properly renormalized, defines a family of radial mollifiers on $\mathbb C$.

\begin{lem}\label{lem: mollifiers ginibre}
Let $\mathcal{K}_L(z,w)=\mathcal{K}_L(|z-w|)$ be the kernel of the infinite Ginibre point process on $\mathbb C$ given by \eqref{kernel_ginibre}. Then, for $p>-2$, 
\[
    h_{p, L}(t)=C_{p,L} \mathcal{K}_L(t)^2t^p, \quad L> 0,
\]
where
\[
C_{p,L}=\pi\Gamma\Big(\frac{p}{2}+1\Big)^{-1}\Big(\frac{1}{\sqrt L}\Big)^{2-p},
\]
is a sequence of radial mollifiers on $\mathbb C$.
\end{lem}
\begin{proof}
For any $\delta>0$, by  \eqref{kernel_ginibre} we have
\[
\begin{split}
\int_{|z|<\delta} h_{p, L}(|z|)\,  dz&=2\pi\int_0^\delta h_{p, L}(t)t\, dt=\frac{2}{\pi}L^2C_{p,L}\int_0^\delta e^{-Lt^2}t^{p+1} \, dt\\
&=2\Gamma\Big(\frac{p}{2}+1\Big)^{-1}\int_0^{\sqrt{L}\delta} e^{-s^2}s^{p+1} ds.
\end{split}
\]
Since $\int_0^\infty e^{-s^2}s^{p+1} ds=\Gamma(1+p/2)/2$, taking the limit as $\delta\to \infty$ of the above expression we get,
\[
    \int_\mathbb{C}h_{p, L}(|z|)dz=1, \quad L>0,
\]
which is $(i)$ in the definition of radial mollifiers \eqref{def_mollifier}, while taking the limit as $L\to \infty$ we get that
\[
    \lim_{L\to\infty}\int_{|z|<\delta} h_{p, L}(|z|)\, dz=1,
\]
which is $(ii)$ in \eqref{def_mollifier}.
\end{proof}

\begin{proof}[Proof of Theorem~\ref{Ginibre-smooth}]
Applying \eqref{equation_variance} to $f \in H^1(\mathbb S^d)$ we get
\[
\begin{split}
    \var\Bigl(\sum_{x\in \Xi_L} f(x)\Bigr) &=\iint_{\mathbb C\times \mathbb C}
\frac{|f(z)-f(w)|^2}{|z-w|^2} \frac{\mathcal{K}_L(|z-w|)^2 |z-w|^2}{2}\, dzdw\\
&=\frac{1}{2\pi}\iint_{\mathbb C\times \mathbb C}
\frac{|f(z)-f(w)|^2}{|z-w|^2}h_{2,L}(|z-w|) \,dzdw,
\end{split}
\]
where $h_{2,L}$ is defined as in Lemma~\ref{lem: mollifiers ginibre}. Then, applying Lemma~\ref{lem: mollifiers ginibre} for $p=2$ and \cite[Theorem 2]{brezis}, and recalling that $K_{2,2}=1/2$, we get the result.
\end{proof}

\begin{proof}[Proof of Theorem~\ref{Ginibre-rough}]
Applying \eqref{equation_variance} to the characteristic function of $A\subset \mathbb C$, we get
\[
\begin{split}
    \frac{1}{\sqrt{L}}\var(n_A)&=\iint_{\mathbb C\times \mathbb C}
\frac{|\chi_A(z)-\chi_A(w)|}{|z-w|} \frac{\mathcal{K}_L(|z-w|)^2 |z-w|}{2\sqrt{L}}\, dzdw\\
&=\frac{1}{4\sqrt{\pi}}\iint_{\mathbb C\times \mathbb C}
\frac{|\chi_A(z)-\chi_A(w)|}{|z-w|}h_{1,L}(|z-w|)\, dzdw,
\end{split}
\]
where $h_{1,L}$ is defined as in Lemma~\ref{lem: mollifiers ginibre}. Then, applying Lemma~\ref{lem: mollifiers ginibre} for $p=1$ and \cite[Theorem 3]{brezis}, and recalling that $K_{2,1}=2/\pi$, we get the result.
\end{proof}

We now come to the finite Ginibre ensemble. In this case, the kernel of the process, given by \eqref{kernel_ginibre_finite}, is not radial, and the results of \cite{brezis} do not apply. We will get the asymptotic behavior of both the smooth and rough linear statistics as a consequence of Theorem~\ref{thm:asymptotically radial mollifiers}. In order to apply the result, we need to show that the squared modulus of the kernel of the process, properly renormalized, defines a family of asymptotically radial $p$-mollifiers.

\begin{lem}\label{lem: mollifiers ginibre finite}
Let $K_n(z,w)$ be the kernel of the finite Ginibre point process on $\mathbb C$ given by \eqref{kernel_ginibre_finite}. Then, for $p>-2$, 
\[
    \varrho_{p, n}(z,w)=C_{p,n} |K_n(z,w)|^2|z-w|^p, \quad n> 0,
\]
where $C_{p,n}$ is defined as in Lemma~\ref{lem: mollifiers ginibre}, is a sequence of asymptotically radial $p$-mollifiers on $\Omega=\lambda\mathbb D$, $\lambda \in (0,1)$.   
\end{lem}

\begin{proof}
Let $h_{p,n}$ be defined as in Lemma~\ref{lem: mollifiers ginibre}. Clearly 
\[
|K_n(z,w)|\le\frac{n}{\pi}e^{-\frac{n}{2}(|z|-|w|)^2}, \quad \mathcal{K}_n(|z-w|)\le \frac{n}{\pi}e^{-\frac{n}{2}(|z|-|w|)^2},
\]
from which follows
\[
|\sqrt{h_n(|z-w|)}+\sqrt{\varrho_n(z,w)}|\le 2C_{p,n}^{1/2}\frac{n}{\pi}e^{-\frac{n}{2}(|z|-|w|)^2}|z-w|^{p/2}.
\]
Moreover,
\[
\begin{split}
    \Big|e^{nz\overline{w}}-\sum_{k=0}^{n-1}\frac{(n z\overline{w})^k}{k!}\Big| &= \Big|\sum_{k=n}^\infty\frac{(n z\overline{w})^k}{k!}\Big|\le \frac{(n |z\overline{w}|)^n}{n!}\sum_{k=n}^\infty\frac{n^{k-n} |z\overline{w}|^{k-n}}{k!}n!\\
    &=\frac{n^n}{n!}|z\overline{w}|^n\sum_{k=0}^\infty\frac{n^k |z\overline{w}|^k}{(n+k)!}n!\le \frac{1}{1-\lambda^2}\frac{n^n}{n!}|z\overline{w}|^n,
\end{split}
\]
where in the last inequality we used  that $(n+k)!\ge n!n^k$ and $|z\overline{w}|<\lambda^2$. By Stirling's estimate, $n^n/n!\leq e^n/\sqrt{n}$, so we can further bound the above term by
\[
\frac{1}{1-\lambda^2}
    \frac{e^n}{\sqrt{n}}|z\overline{w}|^n=\frac{1}{1-\lambda^2}\frac{(|z\overline{w}| e^{1-|z\overline{w}|})^n}{\sqrt{n}}e^{|z\overline{w}| n}\leq \frac{(\lambda^2 e^{1-\lambda^2})^n}{1-\lambda^2}\frac{e^{|z\overline{w}| n}}{\sqrt{n}}.
\]
This implies,
\[
\begin{split}
  \Big|\mathcal{K}_n(|z-w|)-|K_n(z,w)|\Big|&\le\Big|\mathcal{K}_n(|z-w|)-K_n(z,w)\Big|= \frac{n}{\pi}e^{-n\frac{|z|^2+|w|^2}{2}}\Big|e^{nz\overline{w}}-\sum_{k=0}^{n-1}\frac{(n z\overline{w})^k}{k!}\Big|\\
  &\le \frac{\sqrt{n}}{\pi(1-\lambda^2)}(\lambda^2 e^{1-\lambda^2})^ne^{-\frac{n}{2}(|z|-|w|)^2}.
\end{split}
\]
Then,
\[
|\sqrt{h_n(|z-w|)}-\sqrt{\varrho_n(z,w)}|\le \frac{C_{p,n}^{1/2}\sqrt{n}}{\pi(1-\lambda^2)}(\lambda^2 e^{1-\lambda^2})^ne^{-\frac{n}{2}(|z|-|w|)^2}|z-w|^{p/2}.
\]
It follows,
\[
|h_n(|z-w|)-\varrho_n(z,w)|\le 2\Gamma\Big(\frac{p}{2}+1\Big)^{-1}\frac{n^{(p+1)/2}}{\pi(1-\lambda^2)}(\lambda^2 e^{1-\lambda^2})^ne^{-n(|z|-|w|)^2}|z-w|^{p}.
\]
Since $\lambda^2 e^{1-\lambda^2}<1$, we have
\[\lim_{n\to\infty}\int_{B(z,1)} \frac{|\varrho_n(z,w)-h_n(|z-w|)|}{|z-w|^p}\, dw = 0,
\]
and the convergence holds true uniformly on compact sets, since $e^{-n(|z|-|w|)^2}\leq 1$. This shows that $\varrho_n$ fulfills condition (c). Then by (c) and Lemma~\ref{lem: mollifiers ginibre} we immediately have that
\[\lim_{n\to\infty}\int_{B(z,1)} \varrho_n(z,w) \, dw = \lim_{n\to\infty}\int_{B(z,1)}h_n(|z-w|)
\, dw=1,
\]
i.e., that $\varrho_n$ fulfills condition (b). To prove (a), write
\[
\begin{split}
\int_{B(z,\delta)^c} \frac{\varrho_n(z,w)}{|z-w|^p}dw&=\int_{\mathbb C} \frac{\varrho_n(z,w)}{|z-w|^p} dw-\int_{B(z,\delta)} \frac{\varrho_n(z,w)}{|z-w|^p} dw\\
&=C_{p,n}\Big(\int_{\mathbb C} |K_n(z,w)|^2 dw-\int_{B(z,\delta)} \mathcal{K}_n(|z-w|)^2 dw\Big)+o(1),
\end{split}
\]
where in the last equality we used property (c). Now, on one hand 
\[
\int_{B(z,\delta)} \mathcal{K}_n(|z-w|)^2 dw=\frac{2n^2}{\pi}\int_0^\delta e^{-nr^2}r dr=\frac{n}{\pi}(1-e^{-n\delta^2})\sim \frac{n}{\pi}, \quad n\to\infty,
\]
while on the other, by the properties of reproducing kernels,
\[
\int_{\mathbb C} |K_n(z,w)|^2 dw=K_n(z,z),
\]
which is asymptotic to $n/\pi$ as $n\to\infty$ since $|z|<1$ (see Section \ref{sec:ginibre}). This proves that
\[
\lim_{n\to\infty}\int_{B(z,\delta)^c} \frac{\varrho_n(z,w)}{|z-w|^p}dw=0,
\]
which is (a), and completes the proof.
\end{proof}
We are now in the position of applying Theorem~\ref{thm:asymptotically radial mollifiers} to establish the asymptotic behavior of smooth and rough linear statistics of the finite Ginibre ensemble.

\begin{proof}[Proof of Theorem~\ref{finite Ginibre-smooth}]

Let  $f \in L^2(\mathbb C)$ with compact support $K\subset\mathbb D$. By \eqref{equation_variance},
\[
\begin{split}
\var\Big( \sum_{i=1}^N f(x_i)\Big)
&=\iint_{\mathbb C\times \mathbb C}
\frac{|f(z)-f(w)|^2}{|z-w|^2} \frac{K_N(z,w)^2 |z-w|^2}{2}\, dzdw\\
&=\frac{1}{2\pi}\iint_{\mathbb C\times \mathbb C}
\frac{|f(z)-f(w)|^2}{|z-w|^2}\varrho_{2,N}(|z-w|) \,dzdw,
\end{split}
\]
where $\varrho_{2,N}$ is defined as in Lemma~\ref{lem: mollifiers ginibre finite}. Since by Lemma~\ref{lem: mollifiers ginibre finite} the functions $\varrho_{2,N}$ are asymptotically radial $2$-mollifiers on $\Omega=\mathbb \lambda D$, which is an open neighborhood of the support of $f$ for a certain $0<\lambda < 1$ suitably chosen, by Theorem~\ref{thm:asymptotically radial mollifiers}, and recalling that $K_{2,2}=1/2$, we get the result.
\end{proof}

\begin{proof}[Proof of Theorem~\ref{finite Ginibre-rough}]
Let  $A$ be a Borel set such that $\overline{A}\subset \mathbb D$. Applying \eqref{equation_variance} to the characteristic function of $A$, we get
\[
\begin{split}
    \frac{1}{\sqrt{N}}\var(n_A)&=\iint_{\mathbb C\times \mathbb C}
\frac{|\chi_A(z)-\chi_A(w)|}{|z-w|} \frac{K_N(z,w)^2 |z-w|}{2\sqrt{N}}\, dzdw\\
&=\frac{1}{4\sqrt{\pi}}\iint_{\mathbb C\times \mathbb C}
\frac{|\chi_A(z)-\chi_A(w)|}{|z-w|}\varrho_{1,N}(|z-w|)\, dzdw,
\end{split}
\]
where $\varrho_{1,N}$ is defined as in Lemma~\ref{lem: mollifiers ginibre finite}. Since by Lemma~\ref{lem: mollifiers ginibre finite} the functions $\varrho_{1,N}$ are asymptotically radial $1$-mollifiers on $\Omega=\lambda\mathbb D$, which is an open neighborhood of $\overline{A}$ for a suitable chosen $0<\lambda<1$, by Theorem~\ref{thm:asymptotically radial mollifiers}, and recalling that $K_{2,1}=2/\pi$, we get the result.
    
\end{proof}

\subsection{Bessel's process: smooth and rough statistics}

The smooth statistics of the Bessel point process follow directly from the Riemann-Lebesgue theorem and some classical estimates for the Bessel functions.

\begin{proof}[Proof of Proposition~\ref{rough_bessel}]
We know that by \eqref{equation_variance} we need to compute
\[
 \lim_{L\to\infty} \frac{ 2^{d-1}\Gamma\left( \frac{d}{2}+1\right)^2}{L^{d-1}} \int_{\R^d}\int_{\R^d}  \frac {L^{d} J_{d/2}^2(L\pi|x-y|)} {(\pi|x-y|)^d} |f(x)-f(y)|^2\, dxdy.
\]
We take $M$ very big and split the integral in the points near the diagonal $|x-y| < M/L$ and the points far from it. In the points near the diagonal we use that $|K_L(x,y)|\lesssim L^{d/2}$ and thus
\[
\begin{split}
    &\lim_{L\to \infty} \iint_{|x-y|< M/L} \frac {L J_{d/2}^2(L\pi|x-y|)} {(\pi|x-y|)^d} |x-y|^{d+1}\frac{|f(x)-f(y)|^2}{|x-y|^{d+1}}\, dxdy \\
    &\lesssim M^{d+1}\lim_{L\to \infty}  \iint_{|x-y| < M/L}\frac{|f(x)-f(y)|^2}{|x-y|^{d+1}}\, dxdy = 0.
\end{split}  
\]
As for the points far from the diagonal, we use the fact that
\[
 J_\alpha(z) = \sqrt{\frac{2}{\pi z}}\left(\cos \Bigl(z-\frac{\alpha\pi}2 -\frac \pi 4\Bigr)+ \bigO{z^{-1}}\right),
\]
and
\[
\begin{split}
 &\lim_{L\to\infty}  \frac{ 2^{d-1}\Gamma\left( \frac{d}{2}+1\right)^2}{L^{d-1}} \iint_{|x-y|L>M} \frac {L^{d} J_{d/2}^2(L\pi |x-y|)} {(\pi|x-y|)^d} |f(x)-f(y)|^2\, dxdy  \\
 &= \frac{ 2^{d}\Gamma\left( \frac{d}{2}+1\right)^2}{\pi^{d+2}} \Big(\lim_{L\to\infty}\iint_{|x-y|L>M}\!\! \cos^2\Big(L|x-y| -\frac{(d+1)\pi}{4}\Big)\frac{|f(x)-f(y)|^2}{|x-y|^{d+1}}\,dxdy \\
 &+\lim_{L\to\infty}\iint_{|x-y|L>M}\!\! \bigO{\frac 1{L|x-y|}} \frac{|f(x)-f(y)|^2}{|x-y|^{d+1}}\,dxdy\Big).
\end{split}
\]
The second summand is $\bigO{1/M}[f]_{1/2}^2$, so we can discard it taking $M$ very big and
the first summand
\begin{align*}
 \lim_{L\to\infty} & \iint_{|x-y|>M/L} \cos^2\Big(L|x-y| -\frac{(d+1)\pi}{4}\Big)\frac{|f(x)-f(y)|^2}{|x-y|^{d+1}}\,dxdy 
 \\
 &
 = \frac 12\iint_{\R^d\times\R^d}  \frac{|f(x)-f(y)|^2}{|x-y|^{d+1}}\,dxdy,
\end{align*}
by the cosine double angle formula and the Riemann-Lebesgue theorem.
\end{proof}

To establish the asymptotic behavior of the rough statistics of the Bessel process, instead, we can exploit Theorem~\ref{thm:asymptotically radial mollifiers}. Since the kernel of the process is radial in this case,  only conditions (a) and (b) need to be verified, and Theorem~\ref{thm:asymptotically radial mollifiers} reduces to the results proved in \cite{Gounoue}.

\begin{proof}[Proof of Theorem~\ref{Bessel-rough}]
Let $K_L(x,y)=K_L(|x-y|)$ be the Bessel kernel given in \eqref{kernel_bessel}, namely
\[
K_L(|x-y|) = B_{d,L} \frac {J_{d/2}(L\pi|x-y|)} {|x-y|^{d/2}}, \quad B_{d,L}:=\Big(\frac{2}{\pi}\Big)^{d/2} \Gamma\left(\frac{d}{2}+1\right)L^{d/2}.
\]
Applying \eqref{equation_variance} to the characteristic function of $A$ we get
\[
\begin{split}
        \var(n_A)&=\iint_{\R^d\times \R^d}
\frac{|\chi_A(x)-\chi_A(y)|}{|x-y|}\frac{K_L(x,y)^2|x-y|}{2}\, dxdy\\
&=C_{d,L}\iint_{\R^d\times \R^d}
\frac{|\chi_A(x)-\chi_A(y)|}{|x-y|}\varrho_L(|x-y|)\, dxdy,
\end{split}
\]
where
\[
\varrho_L(t):=\frac{1}{2C_{d,L}}K_L(t)^2t=\frac{B_{d,L}^2}{2C_{d,L}}\frac{J_{d/2}^2(L\pi t)}{t^{d-1}}, \quad C_{d,L}=K_d^{-1}C_dL^{d-1}\log L,
\]
and $C_d$ is as in the statement of the theorem.
Then, in view of Theorem~\ref{thm:asymptotically radial mollifiers}, to prove the result it is enough to show that the functions $\varrho_L$ satisfy (a) and (b) in the definition of asymptotically radial $1$-mollifiers.

To see this, begin by observing that, up to multiplication by a constant depending only on $d$, from the bound
$J_{d/2}^2(t)\lesssim 1/t$, we get
\[
\varrho_L(t)\approx \frac{L}{\log L}\frac{J_{d/2}^2(L\pi t)}{(\pi t)^d}t\lesssim \frac{1}{L\log L}\frac{1}{t^d},
\]
from which follows that for any $\delta>0$
\[
\int_{B(x,\delta)^c}\frac{\varrho_L(|x-y|)}{|x-y|}dy\lesssim \frac{1}{L\log L}\int_\delta^\infty\frac{1}{t^2}dt,
\]
which clearly tends to zero as $L\to\infty$. This proves (a), for $p=1$, in the definition of asymptotically radial mollifiers. It only remains to show that for any fixed $\delta>0$, the integral
\begin{equation}\label{split_integral}
       \int_{B(x,\delta)}\varrho_L(|x-y|)\, dy=
\omega_{d-1}\int_0^{1/L}\varrho_L(t)t^{d-1}\, dt+\omega_{d-1}\int_{1/L}^\delta \varrho_L(t)t^{d-1}\, dt, 
\end{equation}
tends to 1 as $L\to\infty$.

For the first integral to the right of \eqref{split_integral}, using that $J_{d/2}^2(L\pi t)\le 1$, we get
\[
\int_0^{1/L}\varrho_L(t)t^{d-1}\, dt\lesssim \frac{L}{\log L}\int_0^{1/L}\, dt
\]
which converges to zero when $L\to +\infty$.

For the second integral in the right hand-side of \eqref{split_integral}, we use the asymptotic estimate
\[
 J_{d/2}(z) = \sqrt{\frac{2}{\pi z}}\left(\cos \Big(z-\frac{(d+1)\pi}4\Big)+\bigO{z^{-1}} \right),
\]
to get
\[
\begin{split}
  \int_{1/L}^\delta \varrho_L(t)t^{d-1}\, dt&=\frac{B_{d,L}^2}{2C_{d,L}}\int_{1/L}^\delta J_{d/2}^2(L\pi t)\, dt\\
  &=\frac{1}{\pi^2L}\frac{B_{d,L}^2}{C_{d,L}}\int_{1/L}^\delta\frac{1}{t}\left(\cos \Big(L\pi t-\frac{(d+1)\pi}4\Big)+\frac{1}{L}\bigO{t^{-1}} \right)^2 dt.
\end{split}
\]
Since $B_{d,L}^2/C_{d,L}\approx L/\log L$, for $L\to\infty$ we get
\[
\begin{split}
  \int_{1/L}^\delta \varrho_L(t)t^{d-1}\, dt&= \frac{1}{\pi^2L}\frac{B_{d,L}^2}{C_{d,L}}\int_{1/L}^\delta\frac{1}{t} \cos^2 \Big(L\pi t-\frac{(d+1)\pi}4\Big)dt+ o(1)\\
  &=\frac{1}{2\pi^2L}\frac{B_{d,L}^2}{C_{d,L}}\int_{1/L}^\delta\frac{1}{t}dt+ o(1),
\end{split}
\]
where in the last line we used the trigonometric identity $2\cos^2(\theta)=1+\cos(2\theta)$ and the Riemann-Lebesgue lemma. Putting all together,
\[
\begin{split}
    \int_{B(x,\delta)}\varrho_L(|x-y|)\, dy&=
\frac{\omega_{d-1}}{2\pi^2L}\frac{B_{d,L}^2}{C_{d,L}}\int_{1/L}^\delta\frac{1}{t}dt+ o(1)=\frac{\omega_{d-1}}{2\pi^2}\frac{\log L}{L}\frac{B_{d,L}^2}{C_{d,L}}+o(1)\\
&=\frac{\omega_{d-1}}{2\pi^2}=1+o(1),
\end{split}
\]
which proves (b) in the definition of asymptotically radial $1$-mollifiers and completes the proof.
\end{proof}

\bibliographystyle{amsalpha}
\bibliography{references}

\end{document}